\documentclass[12pt]{article}

\usepackage{amsthm}
\newtheorem{teo}{Theorem}[section]
\newtheorem{prop}[teo]{Proposition}
\newtheorem{cor}[teo]{Corollary}

\theoremstyle{definition}
\newtheorem{defi}[teo]{Definition}
\newtheorem{rem}[teo]{Remark}
\newtheorem{problem}[teo]{Problem}

\usepackage{amssymb}
\usepackage{amsmath}

\usepackage[all]{xypic}

\usepackage{booktabs}

\usepackage{multirow}
\usepackage{array}

\usepackage{enumerate}

\usepackage[plain]{algorithm}
\usepackage{algorithmic}

\algsetup{linenosize=\scriptsize}

\newcommand{\duepuntiuguale}{\mathrel{\mathop:}=}

\newcommand{\confspace}{X}

\newcommand{\module}{{\cal A}}

\newcommand{\scheme}{A}

\newcommand{\hyperplanesep}[1]{{\cal H}_{#1}}
\newcommand{\featuresep}[1]{\overline{{\cal H}}_{#1}}

\newcommand{\calT}{{\cal T}}
\newcommand{\calV}{{\cal V}}
\newcommand{\calE}{{\cal E}}

\newcommand{\rightarc}[2]{\overrightarrow{#1#2}}
\newcommand{\leftarc}[2]{\overleftarrow{#1#2}}

\newcommand{\path}[2]{P\left(#1,#2\right)}
\newcommand{\dompath}[2]{DP\left(#1,#2\right)}

\newcommand{\derived}[1]{\widetilde{#1}}

\newcommand{\maxcomp}[1]{#1_{\mathrm{MAX}}}

\newcommand{\nodeone}[1]{*=<12pt>[o][F-]{#1}}
\newcommand{\nodetwo}[1]{*=<14pt>[o][F-]{#1}}
\newcommand{\nodethree}[1]{*=<16pt>[o][F-]{#1}}
\newcommand{\nodesquaretwo}[1]{*=<14pt>[F-]{#1}}
\newcommand{\arc}[1]{\ar @{-} #1 |-(0.57){\SelectTips{cm}{}\object@{>}}}
\newcommand{\arccurv}[2]{\ar @{-} #2 #1 |-(0.57){\SelectTips{cm}{}\object@{>}}}
\newcommand{\arccurvshift}[3]{\ar @{-} #2 #1 |-(#3){\SelectTips{cm}{}\object@{>}}}

\newcommand{\riquadrotratt}{\ar@{.} +UL(1.5);+UR(1.5) \ar@{.} +UR(1.5);+DR(1.5)
\ar@{.} +DR(1.5);+DL(1.5) \ar@{.} +DL(1.5);+UL(1.5) }
\newcommand{\riquadrotrattbig}[4]{\ar@{.} #1+UL(1.5);#2+UR(1.5) \ar@{.} #2+UR(1.5);#3+DR(1.5)
\ar@{.} #3+DR(1.5);#4+DL(1.5) \ar@{.} #4+DL(1.5);#1+UL(1.5) }
\newcommand{\freccia}[1]{\ar@{.} #1 |-{\object@{>}}}

\newcommand{\R}[0]{\mathbb{R}}
\newcommand{\Cal}[1]{\mathcal{#1}}
\newcommand{\K}[0]{\mathbb{K}}
\newcommand{\CC}[0]{\mathbb{C}}

\title{Modularity and Optimality in Social Choice}

\author{Gennaro {\sc Amendola}\footnote{Department of Mathematics and Applications, University of Milano-Bicocca (Type~A Research Fellowship, formerly ``E.~De~Giorgi'' grant from the Department of Mathematics of the University of Salento), gennaro.amendola@unimib.it}
\and Simona {\sc Settepanella}\footnote{LEM, Scuola Superiore Sant'Anna, s.settepanella@sssup.it}}
\begin{document}

\maketitle

\begin{abstract}
Marengo and the second author have developed in the last years a geometric model of social choice when this takes place among bundles of interdependent elements, showing that by bundling and unbundling the same set of constituent elements an authority has the power of determining the social outcome.
In this paper we will tie the model above to tournament theory, solving some of the mathematical problems arising in their work and opening new questions which are interesting not only from a mathematical and a social choice point of view, but also from an economic and a genetic one.
In particular, we will introduce the notion of u-local optima and we will study it from both a theoretical and a numerical/probabilistic point of view;
we will also describe an algorithm that computes the universal basin of attraction of a social outcome in $O(M^3\log M)$ time (where $M$ is the number of social outcomes).
\end{abstract}

\begin{center}
{\small\noindent{\bf Keywords}:\\
Social rule, modularity, object, optimum,\\
hyperplane arrangement, tournament, algorithm.}
\end{center}

\begin{center}
\end{center}

\begin{center}
\end{center}

\begin{center}
{\small\noindent{\bf ACM CCS (1998)}:
G.2.2.}
\end{center}

\section*{Introduction}

In~\cite{Arrow} Arrow created {\em modern social choice theory}, a rigorous melding of social ethics and voting theory with an economic flavor.
The central aim of social choice theory is to analyze the aggregation of preferences.
Assume there is a society of $n$ agents indexed by $i = 1,\dotsc,m$.
Each agent has his own well-behaved preference $\succeq_i$ over some space of possibilities (or {\em social outcomes}) $X$, i.e.~a total order on the set $X$.
Let ${\cal P}$ be the set of well-behaved preferences.
The element $(\succeq_1,\dotsc,\succeq_m) \in {\cal P}^m$ is the {\em profile} of a society.
The goal is to put all of these preferences together to come up with a single {\em system of social preferences} or {\em social rule}, i.e.~a total order on the set $X$, to decide matters of policy and to evaluate welfare.
Namely, a {\em social choice function} (or {\em social decision rule})
$$
{\Cal R}: {\cal P}^m \longrightarrow {\cal P}
$$
is needed.

This social decision rule should fulfill the following properties.
\begin{list}{}{\setlength{\itemindent}{-1cm} \setlength{\labelwidth}{1cm}}
\item\textsc{Completeness and Transitivity}\quad
With this economists means that society can make a decision about any social outcome and can rank all social outcomes.
(Obviously, this property is intrinsic in the definition of the function ${\cal R}$.)
\item\textsc{Paretianity}\quad
If everyone unanimously prefers $x$ to $y$ then so should 
society.
\item\textsc{Universal Domain Property}\quad
No matter what kind of wacky preferences people
may have, so long as they are well-behaved, ${\cal R}$ has to be able to 
deal with them. 
In other words there are no restrictions on the profiles of
preferences, i.e. on the elements in ${\cal P}^m$.
\item\textsc{Independence of Irrelevant Alternatives}\quad
Whether or not society prefers $x$ to $y$ does not depend on what people think
of any other {\em irrelevant alternative} $z$.
This can be formally stated by saying that if there are two profiles of
individual preferences $(\succeq_1,\dotsc,\succeq_m)$ and 
$(\succeq'_1,\dotsc,\succeq'_m)$
such that
$$
x \succeq_i y\quad \text{if and only if}\quad x \succeq'_i y,
$$
then
$$
x\ {\cal R}(\succeq_1, \ldots , \succeq_m)\ y\quad \text{if and only if}\quad
x\ {\cal R}(\succeq'_1, \ldots , \succeq'_m)\ y.
$$
\item\textsc{Nondictatorship}\quad
An agent $a_i$ is said to be {\em dictatorial} if, for all $x,y \in X$, whenever $a_i$ prefers $x$ to $y$ society prefers $x$ to $y$, i.e.~${\cal R}$ is the projection on the $i$-th component.
The social decision rule ${\cal R}$ is said to be {\em nondictatorial} if it is not a projection map.
\end{list}
Arrow~\cite{Arrow} proved that such a function does not exist.
Therefore, in order to overcome this problem, in social choice theory it is a customary convention to drop the transitivity request.

Mathematicians and economists studied this problem during the last 50 years with different approaches.
For example, tournament (and, in general, graph) theory turned out to be strictly connected to voting and social choice problems, since Landau started to study this subject~\cite{LandauI,LandauII,LandauIII}.
In the works of Eckmann~\cite{ecman}, Eckmann, Ganea and Hilton~\cite{ecmannb}, and Weinberger~\cite{wein} there has been an ``unexpected application of algebraic topology to a different field of intellectual enterprise,''\footnote{Cited by B.~Eckmann.} i.e. the social choice theory.
Topology is also used to study social choice problems as, for instance, Chichilnisky~\cite{cici1,cici2,cici3} and Baryshnikov~\cite{bary} have done.
Very recently, Saari~\cite{saari94,Saarib} used geometry to analyze the matter of voting.
Moreover, Terao~\cite{terao} introduced an admissible map of chambers of a 
real central arrangement which is a generalization of a social welfare function.

Social choice theory usually assumes that agents are
faced with a set of exogenously given and mutually exclusive alternatives. These
alternatives are ``simple'', in the sense that are one-dimensional objects or, even when
they are multidimensional, they are simply points in some portion of the homogeneous $\mathbb{R}^n$ space and they lack an internal structure that limits the set of possible alternatives.

Many choices in real life situations depart substantially from this simple setting.
Choices are often made among bundles of interdependent elements.
These bundles may be formed in a variety of ways, which in turn affect the selection process of a social outcome.
For instance, in the typical textbook example of social choice, where a group of friends decides what to do for the evening, the choice set is \{movie, concert, restaurant, dinner at home,\ldots\}.
However, at a closer scrutiny, these alternatives are neither primitive nor exogenously given, because they are labels for bundles of elements (e.g. with whom, where, when,\ldots) and the preferences are unlikely to be expressed before the labels get specified in their constituting elements.
Moreover, a member of the group could easily obtain a social outcome close to the one he or she prefers by carefully crafting the objects and possibly designing a new set of objects.
Other examples can be candidates and parties in political elections (which stand for
complex bundles of interdependent policies and personality traits) or packages of policies on which committees and boards are called upon to decide.

In~\cite{MarengoSette} Marengo and the second author develop a model of social choice among
bundles of elements, which they call {\em objects}.
They show that the outcome of the
social choice process is highly dependent on the way these bundles are
formed.
By bundling and unbundling the same set of constituent elements (they call this the {\em object construction power}) an authority may have the power to determine the social outcome.
The object construction power is stronger than the agenda power (i.e.~the power to decide the order on which the social outcomes are decided on), traditionally studied in the literature (for instance, by McKelvey~\cite{mckelvey76}).
Moreover, in their approach, objects decompose the computationally complex search space into quasi-separable subspaces (see Simon~\cite{simon82}), simplifying the computational task and making decisions possible.
They also show that by appropriately designing objects it is possible to break almost all intransitive cycles, which frequently characterize social choice.

In order to formally analyze the properties of a social choice model
with object construction and achieve general results, they use geometric properties of
hyperplane arrangements and link them to graph theory by means of Salvetti's Complex.
In this respect, the model of Marengo and the second author is a novel
contribution to the analysis of the 
relation between discrete problems of social choice and their topological structure.
It provides a bridge between a geometrical representation and a topological one of a social choice
problem to create a more general framework in which the topological space is manipulable
through object construction.

A local study is strictly connected to the geometric structure of the hyperplane arrangement and to the ``local'' structure of the graph, while global properties depend also on the whole graph.
Therefore, in the search for global properties also combinatorial and computational problems arise.

\bigskip

In this paper we tie the model described in~\cite{MarengoSette} to tournament theory.
This new link allows us to get results and opens new problems.
Tournaments are relevant in many fields of science, so they have been greatly studied by many mathematicians, between~1940 and~1970, and almost everything has been done.
We translate the voting and social choice problems, arising in~\cite{MarengoSette}, into new tournament-theory problems interesting for mathematicians too.

Moreover, modularity plays a fundamental role in many Natural Complex Systems~\cite{Modularity} and hence we believe that this model can be applied to other fields of science, e.g.~genetics (see Stadler~\cite{Stadler}).
We plan to go into this subject in a subsequent paper.

In Section~\ref{sec:preliminaries} we will recall some basic notions
on Hyperplane Arrangements, Salvetti's Complex and Tournaments.
In Section~\ref{sect:model} we will give basic notions on social rules, we will describe Marengo and the second author's model, and we will also recall their main results.
In the last part of the section we will define the notion of {\em u-local optimum} and ({\em u-}){\em deepness} of a social outcome.
The former is a compromise between the notion of local optimum and global optimum.
In order to obtain a particular local optimum after the voting process it is enough to have the power of deciding the {\em status quo} from which the voting process starts.
In contrast, u-local optima are characterized by the property of being obtainable after the voting process by means of object construction power only.
This is significant, because it may happen that whoever has the object construction power does not have the power of deciding the {\em status quo} from which the voting process starts.
The deepness and the u-deepness measure the length of voting processes.

In Section~\ref{sec:results} we will give results that tie the model described in~\cite{MarengoSette} to tournament theory.
This link will allow us to prove results in order to compute the universal basin of attraction of a given social outcome (which is a studied and open problem in economics).

In Section~\ref{sec:probability} we will start studying the problem of probability on tournaments with the extra module structure.
We will compute the maximum number of local optima that a given social rule can have and the probability to have a given number of local optima in a two dimensional social rule.
This probability is related to the phenomenon (very important and studied in economics) of the trade-off between {\em decidability} (i.e.~the possibility of reaching some social optimum in a feasible time) and {\em non manipulability} (i.e.~the convergence of the social decision process to a unique global outcome that does not depend upon initial condition and agenda).
Then it would be very interesting to generalize this result to an $n$-dimensional space of features. 

We will also define a function to measure the gain in using Marengo and the second author's model instead of the classical one by means of the probabilities that a social outcome is an optimum in the two models.

In Sections~\ref{sec:algorithm} and~\ref{sec:numerical} we will approach the far more difficult problem of understanding when a local optimum is an u-local or a global one.
In the former section we will give an algorithm that computes the universal basin of attraction of a given social outcome in $O(M^3\log M)$ time, where $M$ is the number of social outcomes.
We point out that this problem deals with both object constructions and agendas.
Since there could be infinitely many agendas, the problem it is not finite {\em a priori}.
It is not difficult to reduce the problem to a finite one, but a simple brute-force algorithm would take far more than exponential time.
This algorithm has been implemented by the first author who has written the computer program~{\tt FOSoR}~\cite{FOSoR}.
In Section~\ref{sec:numerical} we will give numerical data obtained by means of the computer program~{\tt FOSoRStat}~\cite{FOSoRStat} (created by first author) that computes statistics on the number of social rules with a given number of (u-)local optima.

The last section is devoted to one example.
It is treated in some detail because it is the smallest in which all kinds of optima (local, u-local and global) appear.

\paragraph{Acknowledgements}
The authors are grateful to Prof.~Luigi Marengo for his useful comments and corrections.

The first author is grateful to Antonio Caruso for his useful discussions on and help for computer science problems during the beautiful period spent at the Department of Mathematics in Lecce.
He would also like to thank the Department of Mathematics and Applications in Milano for the nice welcome.

\section{Preliminaries}
\label{sec:preliminaries}

\subsection{Hyperplane arrangements and Salvetti's complex}

In this section we will recall some basic notions from the theory of hyperplanes 
arrangements. The interested reader is referred to, for instance, Orlik and Terao~\cite{OT92} for a much more detailed and extended study.

\paragraph{Hyperplane arrangements}
In geometry and combinatorics, an {\em arrangement of hyperplanes} is a finite set $\Cal{A}$ of hyperplanes in a linear, affine, or projective space $S$.
The {\em cardinality} $\left|\Cal{A}\right|$ of the arrangement $\Cal{A}$ is the number of hyperplanes in $\Cal{A}$.

One is normally interested both in the real and in the complex case, hence let $\K$ be either $\R$ or $\CC$ and let $V$ be either $\R^n$ or $\CC^n$.
Thus, given the canonical base
$\{e_1,\dotsc ,e_n\}$ in $V$, each hyperplane $H \in \Cal{A}$ is the kernel of a degree-1 polynomial $\alpha_H \in \K[x_1,\dotsc, x_n]$, defined up to a constant.
The product
\begin{equation*}
\Cal{Q}(\Cal{A})= \prod_{H \in \Cal{A}} \alpha_H
\end{equation*}
is called a {\em defining polynomial} of $\Cal{A}$.

If $\Cal{B}$ is a subset of $\Cal{A}$, it is called a {\em subarrangement} of $\Cal{A}$.
The {\em intersection semilattice} of $\Cal{A}$, denoted by $L(\Cal{A})$, is the set of all non-empty intersections of elements of $\Cal{A}$, i.e.
\begin{equation*}
L(\Cal{A})=\left\{\textstyle{\bigcap_{H \in \Cal{B}} H} \mid \Cal{B} \subseteq \Cal{A} \right\}.
\end{equation*}
These subspaces are called the {\em flats} of $\Cal{A}$.
The set $L(\Cal{A})$ is partially ordered by reverse inclusion.

The {\em complement} of $\Cal{A}$ is defined as
\begin{equation*}
M(\Cal{A}) = V \setminus \bigcup_{H \in \Cal{A}} H .
\end{equation*}
The complement of an arrangement $\Cal{A}$ in $\R^n$ is clearly disconnected.
It is made up of separate pieces called {\em chambers} or {\em regions}, each of which may be either bounded or unbounded.

Each flat of $\Cal{A}$ is also divided into sections by the hyperplanes that do not contain the flat; these sections are called the {\em faces} of
$\Cal{A}$. Chambers are faces because the whole space is a flat. The faces of codimension $1$ may be called the {\em facets} of $\Cal{A}$.
The {\em face semilattice} of an arrangement is the set of all faces, ordered by inclusion.
The arrangement $\Cal A$ said to be {\em essential} if the minimal dimensional flats are points (that we call \emph{vertices} of the arrangement).

Every arrangement $\Cal{A}_\R$ in $\R^n$ also generates an arrangement over $\CC$.
Let $\Cal{Q}(\Cal{A}_{\R})=\prod_{H \in \Cal{A}_{\R}} \alpha_H$ be the defining (real) polynomial of $\Cal{A}_\R$ in $\R^n$. The {\em $\CC$-extended} arrangement $\Cal{A}_\CC$ is the arrangement in $\CC^n$ that consists of the hyperplanes that are the kernel of the polynomials $\alpha_H$ in $\CC^n$ (instead of $\R^n$).
The arrangement $\Cal{A}_\CC$ is also called the {\em complexification} of $\Cal{A}_\R$.

\paragraph{Salvetti's complex}
As shown in~\cite{Sal87}, if the arrangement ${\Cal A}_{\CC}$ is the complexification of a real one ${\Cal A}_{\R}$, there is a regular CW-complex $\Cal{S}({\Cal A}_{\R})$ having the homotopy type of the complement $M({\Cal A}_{\CC})$.
We recall here briefly the construction of this complex, which is called {\em Salvetti's complex}.

Let ${\Cal A}_{\R}=\{H_{\R}\}$ be an essential finite affine hyperplane arrangement in $\R^n$.
Let $M({\Cal A}_{\CC}) =\CC^n \setminus \bigcup_{H_{\R} \in{\cal A}_{\R}} H_{\CC}$ be the complement to the complexified arrangement. The CW-complex $\Cal S({\Cal A}_{\R})$ can be characterized as follows. Let $\mathbf{S}\duepuntiuguale\{F^k\}$ be the stratification of $\R^n$ into facets $F^k$ that is induced by the arrangement~\cite{bourbaki68}, where the exponent $k$ stands for codimension.  Then $\mathbf S$ has a standard partial ordering defined by
\begin{equation*}
F^i \ <_{\mathbf S} F^j \quad \text{if}\quad \operatorname{clos}(F^i)\supset F^j
\end{equation*}
where $\operatorname{clos}(F^i)$ is the closure of $F^i$.

The  $k$-cells of $\Cal{S}({\Cal A}_{\R})$ bijectively correspond to the pairs $[C<_{\mathbf S} F^k]$, where $C$ is a chamber of $\mathbf S$.
A $k$-cell $[C<_{\mathbf S} F^k]$ is in the boundary of a $j$-cell $[D<_{\mathbf S} G^j]$, with $k<j$, if
\begin{itemize}
\item
$F^k<_{\mathbf S} G^j$,
\item
the chambers $C$ and $D$ are contained in the same chamber of the sub-arrangement
$$
\{H_{\R}\in{\Cal A}_{\R}\mid F\subset H_{\R}\}.
$$
\end{itemize}
The previous conditions are equivalent to saying that $C$ is the chamber of ${\Cal A}_{\R}$ ``closest'' to $D$ among those containing $F^k$ in their closure.

It is possible to realize $\cal{S}({\Cal A}_{\R})$ inside $\CC^n$ with explicitly given attaching maps of the cells (see~\cite{Sal87}).

\subsection{Graphs and tournaments}
\label{sec:graphs_tournaments}

We will give here a short summary of graph theory to fix notation.
For a complete discussion we refer the reader to Chartrand and Lesniak~\cite{Chartrand-Lesniak} and Moon~\cite{Moon}.

\paragraph{Graphs}
We will only take oriented simple graphs into account.
Hence, throughout the paper, a {\em graph} will be a pair $(\calV,\calE)$, where $\calV$ is the set of {\em nodes} and $\calE$ is the set of {\em arcs}, such that each pair of nodes $\{p,q\}$ is connected by at most one arc (either $\rightarc{p}{q}$ or $\leftarc{p}{q}$).
If the arc $\rightarc{p}{q}$ (or $\leftarc{q}{p}$) is in $\calE$, the node $p$ is said to {\em dominate} $q$.
A {\em sub-graph} of $(\calV,\calE)$ is a graph $(\calV',\calE')$ such that $\calV'\subset\calV$ and $\calE'\subset\calE$.

A {\em path} $\path{p}{q}$ from $p$ to $q$ is a sequence of arcs of the type
$\rightarc{p}{p_1},\rightarc{p_1}{p_2},\dotsc,\rightarc{p_{k}}{q}$.
A {\em domination path} $\dompath{p}{q}$ from $p$ to $q$ is a sequence of arcs of the type
$\leftarc{p}{p_1},\leftarc{p_1}{p_2},\dotsc,\leftarc{p_{k}}{q}$.
A {\em cycle} $\path{p}{p}$ (resp.~a {\em domination cycle} $\dompath{p}{p}$) is a path (resp.~a domination path) from $p$ to itself.
The {\em length} of a (domination) path is the number of arcs it contains; a cycle of length $k$ is called $k$-cycle.

\paragraph{Tournaments}
A {\em tournament} is a complete graph (i.e.~each pair of nodes $\{p,q\}$ is connected by an arc).
By $\calT$ we will always denote a tournament with $M$ nodes.
A {\em sub-tournament} of $\calT$ is a sub-graph of $\calT$ that is a tournament.

A tournament is said to be {\em reducible} if it is possible to partition its nodes into two non-empty subsets $\calV_1$ and $\calV_2$ in such a way that all the nodes in $\calV_1$ dominate all the nodes in $\calV_2$; otherwise it is called {\em irreducible}.
A tournament is irreducible if and only if each pair of nodes is contained in a cycle.
There is no bound on the length of the cycle, but every node of an irreducible tournament is contained in a $k$-cycle for all $k=3,4,\dotsc,M$.
In particular, any irreducible tournament contains a hamiltonian cycle.
A path is {\em hamiltonian} if it passes through all nodes.
A tournament is {\em transitive} if it contains no cycle.
Every tournament contains a hamiltonian path; if the tournament is transitive the hamiltonian path is unique.

An {\em irreducible component} $\calT_i$ of $\calT$ is a maximal irreducible sub-tournament of $\calT$.
The nodes of these irreducible components form a partition of the nodes of $\calT$.
Moreover, all the nodes of a component $\calT_i$ either dominate or are dominated by all the nodes of another component $\calT_j$.
The transitive tournament $\derived{\calT}$ whose nodes are the irreducible components of $\calT$ and whose arcs are deduced by any arc between the two irreducible components (see Figure~\ref{fig:derived_tournament}) is called the {\em condensation} of $\calT$.
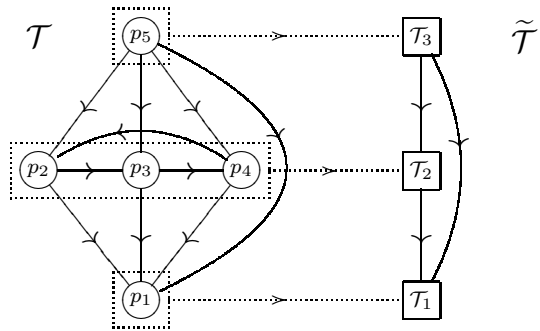
\begin{figure}[t]
  \begin{center}
  $\xymatrix@R=35pt{
        {\mbox{$\calT$}} & \nodetwo{p_5} \arc{[ld]} \arc{[rd]} \arc{[d]}
          \arccurvshift{[dd]}{@/^55pt/}{.38}
          \riquadrotratt \freccia{+R(1.5);[rrr]} &
        & & \nodesquaretwo{\calT_3} \arc{[d]} \arccurvshift{[dd]}{@/^15pt/}{.4}
        & {\mbox{$\derived{\calT}$}}
        \\
        \nodetwo{p_2} \arc{[r]} \arc{[rd]} & \nodetwo{p_3} \arc{[r]} \arc{[d]} & \nodetwo{p_4}
          \arccurvshift{[ll]}{@/_15pt/}{.65} \arc{[ld]}
          \riquadrotrattbig{[ll]}{[]}{[]}{[ll]} \freccia{+R(1.5);[rr]}
        & & \nodesquaretwo{\calT_2} \arc{[d]}
        \\
        & \nodetwo{p_1} \riquadrotratt \freccia{+R(1.5);[rrr]} &
        & & \nodesquaretwo{\calT_1}
        }$
  \end{center}
  \caption{The condensation of a tournament.}
  \label{fig:derived_tournament}
\end{figure}
Without lack of generality, we will choose the subscripts so that if $i>j$ then $\calT_i$ dominates $\calT_j$.
The maximal component of $\calT$ will be denoted by $\maxcomp{\calT}$.

\paragraph{Score}
The number of nodes dominated by a node $p$ is called the {\em score} of $p$.
The sequence of the scores of the the nodes of a tournament is called the {\em score sequence} of $\calT$.
Up to a relabeling of the nodes, we can suppose that the score sequence of $\calT$ is non-decreasing.
A tournament is transitive if and only if its score sequence is $0,1,\dotsc,M-1$.
A non-decreasing sequence $s_1,s_2,\dotsc,s_M$ of nonnegative integers is the score sequence of a tournament if and only if $\sum_{i=1}^{k}s_i\geqslant\binom{k}{2}$ for each $k<M$ and $\sum_{i=1}^{M}s_i=\binom{M}{2}$ hold.
A tournament is irreducible if and only if all the $M-1$ inequalities above are strict.

In order to find the irreducible components (and hence the condensation) of $\calT$, the following very simple algorithm, having complexity $O(M^2)$, can be applied.
\begin{enumerate}
\item
Find the smallest $k$ such that $\sum_{i=1}^{k}s_i=\binom{k}{2}$; the sub-tournament $\calT_1$ made up of the $k$ nodes with smallest score is an irreducible component of $\calT$.
\item
Remove $\calT_1$ from $\calT$ and repeat Step~1 until no node is left.
\end{enumerate}
Note that if the aim is to find the maximal component $\maxcomp{\calT}$, one can start ``from above'' and take into account $M-1-s_i$ instead of $s_i$ (i.e.~the number of nodes that dominate the $i$-th node), so that only one step is needed.
Another algorithm finding the irreducible components of any graph and having complexity $O(M^2)$ is shown in Kocay and Kreher~\cite{Kocay-Kreher}.

\paragraph{Cycles}
A measure of how far a tournament is from being transitive is the number of 3-cycles.
If the score sequence of $\calT$ is $s_1,s_2,\dotsc,s_M$, then the number of $3$-cycles is at most $\binom{M}{3}-\sum_{i=1}^{M}\binom{s_i}{2}\leqslant
\left\{\begin{array}{ll}\frac{M^3-M}{24} & \mbox{if $2\nmid M$}\\[1pt]
  \frac{M^3-4M}{24} & \mbox{if $2\mid M$}\end{array}\right.$.
If $\calT$ is irreducible, the $3$-cycles are at least $M-2$.

\paragraph{Number of tournaments}
The number of tournaments with $M$ nodes (up to relabeling) is $T(M)=\displaystyle{\sum_{(d)}\frac{2^D}{N}}$, where
\begin{itemize}
\item
$(d)=(d_1,d_2,\dotsc,d_M)$ is multi-index with $d_{2i}=0$, $d_{2i+1}\geq0$ and $\sum_{i=1}^{M}i\cdot d_i=M$,
\item
$D$ is $\frac12\left(\sum_{i,j=1}^{M}d_id_j\gcd(i,j)-\sum_{i=1}^{M}d_i\right)$,
\item
$N$ is $\prod_{i=1}^M i^{d_i}d_i!$.
\end{itemize}
The values of $T(M)$ for $M\leqslant12$ are given in Table~\ref{tab:num_tournaments}.
\begin{table}[t]
  \begin{center}
  \begin{small}
  \begin{tabular}{r@{\extracolsep{10pt}}rcr@{\extracolsep{10pt}}rcr@{\extracolsep{12pt}}r}
  \toprule
  $M$ &   $T(M)$ & & $M$ &   $T(M)$ & & $M$ &   $T(M)$ \\
  \cmidrule{1-2}\cmidrule{4-5}\cmidrule{7-8}
  1 & 1 & & 5 &   12 & &  9 &       191536 \\
  2 & 1 & & 6 &   56 & & 10 &      9733056 \\
  3 & 2 & & 7 &  456 & & 11 &    903753248 \\
  4 & 4 & & 8 & 6880 & & 12 & 154108311168 \\
  \bottomrule
  \end{tabular}
  \end{small}
  \end{center}
  \caption{The number of tournaments.}
  \label{tab:num_tournaments}
\end{table}
As $M$ tends to infinity, $T(M)\rightarrow\infty$ and $T(M)\sim\displaystyle{\frac{2^{\binom{M}{2}}}{M!}}$ hold.

The probability $P(M)$ that a tournament with $M$ nodes is irreducible can be computed recursively by the formula
$$
P(M)=1-\sum_{i=1}^{M-1}\binom{M}{i}\frac{P(i)}{2^{t(M-t)}}.
$$
The values of $P(M)$ for $M\leqslant16$ are given in Table~\ref{tab:prob_irred_tournament}.
\begin{table}[t]
  \begin{center}
  \begin{small}
  \begin{tabular}{r@{\extracolsep{10pt}}lcr@{\extracolsep{10pt}}lcr@{\extracolsep{10pt}}lcr@{\extracolsep{10pt}}l}
  \toprule
  $M$ &   $P(M)$ & & $M$ &   $P(M)$ & & $M$ &   $P(M)$ & & $M$ &   $P(M)$ \\
  \cmidrule{1-2}\cmidrule{4-5}\cmidrule{7-8}\cmidrule{10-11}
  1 & 1     & & 5 & 0.53125  & &  9 & 0.931702 & & 13 & 0.993671 \\
  2 & 0     & & 6 & 0.681152 & & 10 & 0.961589 & & 14 & 0.996587 \\
  3 & 0.25  & & 7 & 0.799889 & & 11 & 0.978720 & & 15 & 0.998171 \\
  4 & 0.375 & & 8 & 0.881115 & & 12 & 0.988343 & & 16 & 0.999024 \\
  \bottomrule
  \end{tabular}
  \end{small}
  \end{center}
  \caption{The probability that a tournament is irreducible.}
  \label{tab:prob_irred_tournament}
\end{table}
As $M$ tends to infinity, $P(M)\rightarrow1$ and $P(M)\sim1-\displaystyle{\frac{M}{2^{M-2}}}$ hold.

\section{Definitions and structure of the model}
\label{sect:model}

\paragraph{Social decision rules}

Consider a population of $\nu$ {\em agents}. Each agent $i$ is characterized by a 
{\em system of transitive preferences} $\succeq_i$ over the set of
social outcomes $X$.
The set of systems of transitive preferences $\succeq$ is denoted by $\Cal P$.
A {\em social decision rule} $\Cal R$ is a function:
\begin{equation*}
\begin{matrix}
\Cal R: & \Cal P^{\nu} & \longrightarrow & \overline{\Cal P} \\
 & (\succeq_1 ,\dotsc,\succeq_\nu) & \longmapsto & \succeq_{\Cal R (\succeq_1
,\dotsc,\succeq_\nu)}
\end{matrix}
\end{equation*}
which determines a {\em system of social preferences} or {\em social rule} $\succeq_{\Cal R
  (\succeq_1 ,\dotsc,\succeq_\nu)}$ from the preferences of $\nu$
individual agents.
With $\overline{\Cal P}$ we denote the set of systems of (non-necessarily transitive) social preferences; as a matter of fact, we note that the social rule $\succeq_{\Cal R (\succeq_1 ,\dotsc,\succeq_\nu)}$ is not, in general, transitive anymore.

If $\Delta$ is the diagonal of the cartesian product $X \times X$,
the element $\succeq_{\Cal R} \in \overline{\Cal P}$ defines a subset
\begin{equation*}
\Cal{Y}_{1,\succeq_{\Cal R}}=\{ (x,y)\in X\times X \setminus \Delta \mid 
x \succeq_{\Cal R} y \}
\end{equation*}
and the set of {\em relevant} social outcomes
\begin{equation*}
\Cal{Y}_{0,\succeq_{\Cal R}}=\{x \in X \mid \exists y\in\confspace\ \text{such that}\ (x,y) \in
\Cal{Y}_{1,\succeq{\Cal R}}\ \text{or}\ (y,x) \in \Cal{Y}_{1,\succeq_{\Cal R}} \}.
\end{equation*}
If $\Cal{Y}_{0,\succeq_{\Cal R}}$ is the whole $X$, the social rule is said to be {\em complete}.
A complete social rule is said to be {\em strict} if for each pair of social outcomes $x$ and $y$ the two conditions $x \succeq_{\Cal R} y$ and $y \succeq_{\Cal R} x$ are mutually exclusive (i.e.~either the social outcome $x$ is preferred to the social outcome $y$ or the converse holds).
For the sake of simplicity, we will consider only strict social rules.
This restriction is almost always unnecessary, but it simplifies both the investigation and the presentation.
Therefore, from now on, $\succ$ will always denote a complete strict social rule; unless explicitly stated, it will be considered as fixed.
For the sake of shortness, we will always drop the words ``complete'' and ``strict''.

\paragraph{The graph}
The sets $\Cal{Y}_{0,\succ}$ and $\Cal{Y}_{1,\succ }$ are, respectively, the sets of nodes and arcs of a graph $\Cal{Y}_{\succ}$.
Two nodes $x$ and $y$ in $\Cal{Y}_{0,\succ}$ are connected by an 
arc if $(x,y) \in \Cal{Y}_{1,\succ}$ or 
$(y,x) \in \Cal{Y}_{1,\succ}$; the orientation is from $x$ 
to $y$ in the former case and from $y$ to $x$ in the latter.
For the sake of simplicity, we will use the same symbol $x$ for the nodes of $\Cal{Y}_{\succ}$ and $(x,y)$ for its arcs.
We decided to use different notations for graphs in regard of the relevance of different theories.
Note that the completeness assumption on social rules guarantees that the graph $\Cal{Y}_{\succ}$ is connected.

A cycle
\begin{equation*}
(x_1,x_2),(x_2,x_3),\dotsc ,(x_h,x_1)
\end{equation*}
in the graph $\Cal{Y}_{\succ}$ corresponds to a
cycle {\em \`{a} la} Con\-dor\-cet-Arrow~\cite{Condorcet-Arrow}, i.e.~to the sequence
\begin{equation*}
x_1 \succ x_2 \succ \dotsb \succ x_h \succ x_1.
\end{equation*}

\paragraph{Features}
Let $F=\{f_1,\dotsc ,f_n\}$ be a bundle of elements, said {\em features}, the $i$-th of which takes $m_i$ values, i.e.~$f_i \in \{0,1,2,\dotsc, m_i-1\}$ with $i=1,\dotsc,n$.
Denote by $m=(m_1,\dotsc,m_n)$ the multi-index of the numbers of values of the features.
From now on, a {\em social outcome} (or {\em configuration}) will be an $n$-uple $(v_1,\dotsc,v_n)$ of values such that $0\leqslant v_i<m_i$.
For the sake of shortness, it will be also denoted by $v_1\dotsm v_n$.
The set of all social outcomes will be denoted by $\confspace$.
The cardinality of $\confspace$ is $\prod_{i=1}^n m_i$ and will be denoted by $M$.

\paragraph{The hyperplane arrangement}
There is a correspondence~\cite{MarengoSette} between
the set $X$ of social outcomes and the set $\Cal{C}$ of the chambers of the 
arrangement
\begin{equation*}
\Cal A_{n,m}=\left\{H_{i,j} \mid \alpha_{H_{i,j}}=\lambda_i-j\right\}_{{1 \leqslant i \leqslant n} \atop {0 \leqslant j < m_i-1} }.
\end{equation*}
Namely, $x=v_1\dotsm v_n$ corresponds to the chamber $C$ that contains the open 
set
\begin{equation*}
\left\{(\lambda_1,\dotsc ,\lambda_n) \in \R^n \mid v_j-1 < \lambda_j < v_j,\ j=1, \dotsc ,n\right\}.
\end{equation*}

\paragraph{Salvetti's complex}
There is a correspondence~\cite{MarengoSette} between the oriented graph $\Cal{Y}_{\succ}$ and a subcomplex of the $1$-skeleton of Salvetti's complex $\Cal{S}(\Cal A_{n,m})$ as follows.
Namely, there is a one-to-one correspondence between the $0$-skeleton $\Cal{S}_0(\Cal A_{n,m})$ and the set of chambers in $\Cal A_{n,m}$, i.e.~the set of social outcomes $X$ by means of the correspondence above.
The generators of the $1$-skeleton can be described as
$$
\Cal{S}_1(\Cal A_{n,m})=\{(x,y) \in X \times X \setminus \Delta \mid x\ \text{and}\ y\ \text{are adjacent}\}
$$
where two chambers $C$ and $D$ are said to be {\em adjacent} if they 
are separated by only one hyperplane.

Given a subset of consecutive elements
$$
\{(x_1,x_2),(x_2,x_3),\dotsc, (x_{k-2},x_{k-1}),(x_{k-1},x_k)\}
$$
in $\Cal{S}_1(\Cal A_{n,m})$ their {\em formal sum} is
$$
(x_1,x_k)=\sum_{j=1}^{k-1} (x_j,x_{j+1}).
$$
It follows that given a social rule $\succ$ any arc $(x,y) \in \Cal{Y}_{1,\succ}$ can be written as a formal sum of a minimal number of consecutive elements in $\Cal{S}_1(\Cal A_{n,m})$. The number of elements is exactly the number of hyperplanes that separate the two social outcomes $x,y\in\confspace$.

Let $(x,y) \in \Cal{Y}_{1,\succ}$ be an arc given by a formal sum
with coefficients $1$ of arcs that are in $\Cal{Y}_{1,\succ}$.
If the social rule is transitive the arc can be deleted, because it can be reconstructed by means of the other arcs.

\begin{rem}
Saari has greatly contributed to establishing
general geometric representations of voting models and voting
paradoxes~\cite{saari94,saari00a,saari00b}.
Salvetti's complex is a CW-complex in $\CC^n$, but it has an underlying real
structure which is a purely simplicial complex.
Moreover, vertices in
this complex can be freely chosen inside each chamber.
This structure can be used in order to recast and
generalize some existing geometric models of voting such as those
provided by Saari.
\end{rem}

\paragraph{Objects schemes}
Given a non-empty subset $I\subseteq\{1,\dotsc, n\}$, the {\em object} $\Cal{A}_I$ is the subset
\begin{equation*}
\Cal{A}_I=\{H_{i,j}\}_{{i \in I} \atop {0 \leqslant j < m_i-1}}
\end{equation*}
of the arrangement $\Cal A_{n,m}$.
The cardinality of $\Cal{A}_I$ is called {\em size} of the object 
$\Cal{A}_I$ and is denoted by $|\Cal{A}_I|$.
The complement of a set $I$ in $\{1,\dotsc, n\}$ will be denoted (as usual) by $I^c$, hence the complement $\Cal{A}_I^c=\Cal A_{n,m} \setminus \Cal{A}_I$ of the arrangement $\Cal{A}_I$ in $\Cal A_{n,m}$ turns out to equal $\Cal{A}_{I^c}$.
The {\em object instantiation} $x(\Cal{A}_I)$ of a social outcome $x$ is
the chamber of the subarrangement $\Cal{A}_I$ that contains the chamber 
corresponding to $x$.

An {\em objects scheme} is a set of objects $A=\{\Cal{A}_{I_1},\dotsc, \Cal{A}_{I_k}\}$ such that $\bigcup_{j=1}^k I_j=\{1,\dotsc,n\}$. 
Note that the sets $I_j$ may have non-empty intersection.
The {\em size} of an objects scheme is the size of its largest object,
$$
| A | = \mbox{max}\{| \Cal{A}_{I_1} | ,\dotsc,| \Cal{A}_{I_k} | \}.
$$

\paragraph{Neighbors of a social outcome}
Let be given an objects scheme $A=\{\Cal{A}_{I_1},\dotsc, \Cal{A}_{I_k}\}$.
A social outcome $y$ is said to be a {\em preferred neighbor} of a social outcome $x$ with respect to an object $\Cal{A}_{I_h} \in A$ if the following conditions hold:
\begin{enumerate}[1)]\label{vicini}
\item $y \succ  x$,
\item $y(\Cal{A}_{I_h^c})=x(\Cal{A}_{I_h^c})$, i.e.~$x$ and $y$ belong to the same chamber of the arrangement $\Cal{A}_{I_h^c}$,
\item\label{cond:def_pref_neigh} $y(\Cal{A}_{I_h}) \neq x(\Cal{A}_{I_h})$, i.e.~$x$ and $y$ belong to different chambers of the arrangement $\Cal{A}_{I_h}$.
\end{enumerate}
Note that Condition~\ref{cond:def_pref_neigh} is a direct consequence of the first two, but we have left it for the sake of consistency with the non-strict case.
The set of all preferred neighbors of the social outcome $x$ with respect to $\Cal{A}_{I_h} \in A$ is denoted by $\Phi(x,\Cal{A}_{I_h})$.
The set of all preferred neighbors of the social outcome $x$ is denoted by $\Phi(x,A)=\bigcup_{j=1}^k \Phi(x,\Cal{A}_{I_j})$.

A social outcome $y\in\Phi(x,\Cal{A}_{I_h})$ is said to be a {\em best neighbor} of a social outcome $x$ with respect to an object $\Cal{A}_{I_h} \in A$ if
$$
y \succ w \quad \forall w \in \Phi(x,\Cal{A}_{I_h}).
$$
The set of all best neighbors of the social outcome $x$ with respect to $\Cal{A}_{I_h} \in A$ is denoted by $B(x,\Cal{A}_{I_h})$.
\begin{rem}
Obviously, $B(x,\Cal{A}_{I_h}) \subseteq \Phi(x,\Cal{A}_{I_h})$ holds.
Moreover, either $B(x,\Cal{A}_{I_h})$ is empty, or $B(x,\Cal{A}_{I_h})$ contains one social outcome only.
Even if this notation seems to be useless, we use it to follow the literature-customary convention; indeed, if one takes also non-strict social rules into account, the set $B(x,\Cal{A}_{I_h})$ may contain more than one social outcome.
\end{rem}
The set of all best neighbors of the social outcome $x$ is denoted by $B(x,A)=\bigcup_{j=1}^k B(x,\Cal{A}_{I_j})$.

A {\em domination path $DP(x,y,A) $ through $A$, starting from $x$ and ending in 
$y$}, is a sequence of best neighbors with respect to objects in $A$, 
i.e.~a sequence
$$
x=x_0 \prec x_1 \prec \dotsb \prec x_s=y
$$
such that there exist objects $\Cal{A}_{I_{h_1}}, \dotsc, \Cal{A}_{I_{h_s}} \in A$ with $x_i \in B(x_{i-1},\Cal{A}_{I_{h_i}})$ for all $1 \leqslant i \leqslant s$.

A social outcome $y$ is said to be {\em reachable from $x$ with respect to an objects scheme $A$} if there exists a domination path $DP(x,y,A)$.
A social outcome $x$ is said to be a {\em local optimum for $A$} if $\Phi(x,A)$ is empty.

\paragraph{Agenda}
Let $A=\{\Cal{A}_{I_1},\dotsc,\Cal{A}_{I_k}\}$ be an objects scheme.
An {\em agenda $\alpha$ of $A$} is an ordered $t$-uple of indices $(h_1,\dotsc,h_t)$ with $t\geq k$ such that $\{h_1,\dotsc, h_t\} = \{1,\dotsc, k\}$.
An agenda $\alpha$ states the order in which the objects $\Cal{A}_{I_i}$ are decided upon.
In the model of~\cite{MarengoSette} the agenda is repeated over and over again 
until either a local optimum or a domination cycle is reached.
The ordered $t$-uple of objects $(\Cal{A}_{I_{h_1}},\dotsc,\Cal{A}_{I_{h_t}})$ is denoted by $A_{\alpha}$.
The set of all possible agendas of $A$ is denoted by $\Lambda(A)$.

Let $\alpha = (h_1,\dotsc,h_t)$ be an agenda.
A domination path
$$
x_0 \prec x_1 \prec \dotsb \prec x_s
$$
is said to be {\em ordered along $\alpha$} if
$$
x_i \in B(x_{i-1},\Cal{A}_{I_{h_q+1}})
$$
where $h_q$ is the remainder of the division of $i-1$ by $t$.
Such a domination path will be denoted by $DP(x_0,x_s,A_{\alpha})$.

A domination path is said to be {\em maximal} if it ends in either a local optimum or a limit domination cycle.
More precisely, either $x_s$ is a local optimum or $x_{s-t}$ belongs to $B(x_s,\Cal{A}_{I_{h_s+1}})$, where $h_s$ is the remainder of the division of $s-1$ by $t$.
Note that in the first case we do not require that $x_{s-1}$ is different from $x_s$, so there is no control on the number of times that $x_s$ appears at the end of the domination path.
Also in the second case, there is no control on the number of times that the domination cycle 
$$
x_{s-t} \prec \dotsb \prec x_s
$$
appears at the end of the domination path.
In the first case, we will say that the domination path {\em ends up in $x_s$}.

\paragraph{Basin of attraction}
The {\em basin of attraction $\Psi (x,A)$ of a social outcome $x$ with respect to an objects scheme $A$} is 
the set of the social outcomes $y$ such that there exists a maximal
domination path $DP(y,x,A)$ that ends up in $x$.
\begin{rem}\label{rem:basin_local}
Note that $\Psi (x,A)$ is empty if and only if $x$ is not a local optimum for $A$.
\end{rem}

The {\em ordered basin of attraction $\Psi(x,A_{\alpha})$ of $x$
  with respect to an agenda $\alpha$ of $A$} is the set of the social
outcomes $y$ such that there exists a maximal domination path
$DP(y,x,A_{\alpha})$ that ends up in $x$.
Clearly, we have
$$
\Psi (x,A)= \bigcup_{\alpha \in \Lambda(A)} \Psi (x,A_{\alpha}).
$$

\paragraph{Global optima}
A social outcome $z \in X$ is said to be a {\em global optimum for an agenda $\alpha$} if $\Psi(z,A_{\alpha})=X$ holds.
It is said to be a {\em global optimum for the objects scheme $A$} if and only if $\Psi(z,A_{\alpha})=X$ holds for all agendas $\alpha \in \Lambda(A)$, i.e.~it is a global optimum for all the agendas of $A$.

Local and global optima strictly depend on the choice of the objects scheme $A$.
In~\cite{MarengoSette} the authors prove that object construction power is, 
in some sense, stronger than agenda power, i.e.~they prove
$$
\Psi (z,A) \neq \emptyset \quad \Longleftrightarrow \quad
\Psi (z,A_{\alpha}) \neq \emptyset\ \text{for all}\ \alpha \in \Lambda(A).
$$

\paragraph{Separating hyperplanes and distance between social outcomes}
Let $x$ and $y$ be two social outcomes.
They are said to be {\em separated by an hyperplane $H \in \Cal A_{n,m}$} if $H$ separates the chambers $C_x$ and $C_y$.
In this case, the notation $x \mid H \mid y$ will be used.
Moreover, $x$ and $y$ are said to be {\em prominently separated} if there 
exist two hyperplanes $H_{i_1,j_1}, H_{i_2,j_2} \in \Cal A_{n,m}$ with $i_1 \neq i_2$ (i.e.~non-parallel) such that $x \mid H_{i_1,j_1} \mid y$ and $x \mid H_{i_2,j_2} \mid y$ hold.
We will say that $x$ and $y$ are {\em separated by the feature} $f$ if the value of the feature $f$ of $y$ differs from that of the feature $f$ of $x$.
The set of the features that separates $x$ and $y$ is denoted by $\overline{\Cal{H}}_{x,y}$.

The {\em distance} between $x$ and $y$ is the minimum number of hyperplanes that separate $x$ and $y$.
The {\em prominent distance} $d_p(x,y)$ is the number of features that separate $x$ and $y$, i.e.~$\#\overline{\Cal{H}}_{x,y}$.
Note that $d_p(x,y)$ equals the minimum number of hyperplanes that prominently separate $x$ and $y$.

Recall that, by definition, if $H_{i,\bar{\jmath}}$ belongs to the object $\Cal{A}_I$ for some $\bar{\jmath}$, then $H_{i,j}$ belongs to $\Cal{A}_I$ for all $0 \leqslant j < m_i-1$.
Therefore, the subarrangement
$$
\Cal{H}_{x,y}=\left\{H_{i,j}\in\Cal{A}_{n,m} \mid i\in\overline{\Cal{H}}_{x,y},\ 0 \leqslant j < m_i-1\right\}
$$
of $\Cal{A}_{n,m}$ has been considered.
Note that, if we have $d_p(x,y)=1$ and $d(x,y)>1$, all the hyperplanes in $\Cal{H}_{x,y}$ are parallel.
\begin{rem}
The sets $\Cal{H}_{x,y}$ and $\overline{\Cal{H}}_{x,y}$ are strictly interconnected.
For instance, we will use the fact that $\Cal{H}_{x,y}$ is contained in $\Cal{H}_{z,w}$ if and only if 
$\overline{\Cal{H}}_{x,y}$ is contained in $\overline{\Cal{H}}_{z,w}$.
\end{rem}

In~\cite{MarengoSette} the authors prove the following result.
\begin{teo}\label{principio}
Let $z$ be a social outcome.
Then, there exists an objects scheme $A_z$ for which $z$ is a local optimum if and only if the inequality $d_p(w,z)>1$ holds for any social outcome $w$ with $w\succ z$.
\end{teo}

The previous theorem explains also the reason of the choice of the name ``local optimum''.
Namely, a social outcome $z$ is a local optimum for an objects scheme $A$ if and only if any social outcome $x$ such that $d_p(x,z)=1$ belongs to $\Psi(z,A)$.

A social outcome $z$ is said to be {\em free} if and only if the inequality $d_p(w,z)>1$ holds for any social outcome $w$ with $w\succ z$. Thus, by means of Theorem~\ref{principio}, we have that $z$ is a local optimum for an objects scheme $A_z$ if and only if $z$ is free.

\begin{problem}
An interesting question, pointed out in~\cite{MarengoSette}, is to understand when a local optimum $z$ for an objects scheme $A$ is a global optimum, i.e.~when there exists an agenda $\alpha$ of $A$ such that the basin of attraction $\Psi(z,A_{\alpha})$ is the whole $X$ and when this is true for all agendas $\alpha \in \Lambda(A)$.
\end{problem}

In~\cite{MarengoSette} the authors prove the following.
\begin{teo}\label{pglob1}
Let $z$ be a free social outcome.
Then, there exists an objects scheme $A_z$ such that $\Phi(z,A_z)= \emptyset$ and $\Phi(x,A_z) \neq \emptyset$ for all free social outcomes $x$ if and only if the condition
\begin{equation}\label{eq:cond_hyper}
\exists y \succ x\ \text{such that}\ \Cal{H}_{w,z} \nsubseteq \Cal{H}_{x,y} \quad \forall
w \succ z
\end{equation}
holds for all free $x$.
\end{teo}
The equivalent conditions on the free social outcome $z$ above are necessary for $z$ to be a global optimum.

\paragraph{Universal basin of attraction and u-local optima}
Let $\Pi(\Cal{A}_{n,m})$ be the set of all possible objects schemes in $\Cal{A}_{n,m}$.
The {\em universal basin of attraction} of a social outcome $z \in X$ is the set 
$$
\Psi(z) = \bigcup_{A \in \Pi(\Cal{A}_{n,m})} \Psi(z,A),
$$
i.e.~the set of all the social outcomes $x$ such that there exists an objects scheme through which there is a domination path starting from $x$ and ending up in $z$.

By virtue of Remark~\ref{rem:basin_local} and Theorem~\ref{principio}, the universal basin of attraction of the social outcome $z$ is non-empty if and only if $z$ is {\em free}.
\begin{defi}
A social outcome $z$ is said to be an {\em u-local optimum} if its universal basin of attraction $\Psi(z)$ is the whole set of social outcomes $\confspace$.
\end{defi}

\begin{rem}\label{rem:glob_uloc_loc}
A global optimum is necessarily an u-local optimum, and an u-local optimum is necessarily a local optimum for at least one objects scheme.
\end{rem}

Let $x$ and $z$ be social outcomes.
In~\cite{MarengoSette}, when $z$ is free, the authors consider the set
$$
G_x^z=\{y \succ x \mid \Cal{H}_{w,z} \nsubseteq \Cal{H}_{x,y} \quad \forall w \succ z \mbox { and } B(x, \Cal{H}_{x,y}) \neq \emptyset \}
$$
and prove that if $x$ is in the universal basin of attraction of $z$ then $G_x^z \neq \emptyset$.
For the sake of completeness, we will define $G_x^z$ to be $\emptyset$ if $z$ is not free.
\begin{rem}
Suppose $z$ is free.
The set $G^z_x$ is non-empty if and only if there exists an objects
scheme $\scheme_z$ such that $\Phi(z,\scheme_z)=\emptyset$ and
$\Phi(x,\scheme_z)\neq\emptyset$ hold; i.e. if and only if $x$ satisfies Condition~\eqref{eq:cond_hyper} of Theorem~\ref{pglob1}.
\end{rem}

Suppose now that $x$ is a social outcome such that $G_x^z$ is non-empty (so $z$ is free).
If $B(x,\Cal{H}_{x,y})$ is non-empty, its cardinality is one.
The only element of $B(x,\Cal{H}_{x,y})$ will be denoted by $b_{x,y}$.
In~\cite{MarengoSette} the authors consider the set
$$
BG^z_x=\{b_{x,y} \mid y \in G_x^z\} \subseteq G_x^z,
$$
the sets
\begin{align*}
E_0^z &=\{z\}, \\
E_1^z &=\{x \in X \setminus \{z\} \mid z \in BG_x^z\}, \\
E_2^z &=\{ x \in X \setminus \cup_{i=0}^1 E_i^z \mid E_1^z \cap BG_x^z \neq\emptyset\}, \\
& \quad\quad\quad\quad\quad\quad\quad\vdots \\
E_h^z &=\{ x \in X \setminus \cup_{i=0}^{h-1}E_i^z \mid
E_{h-1}^z \cap BG_x^z \neq \emptyset\}, \\
E_{h+1}^z &=\{ x \in X \setminus \cup_{i=0}^{h}E_i^z \mid
E_{h}^z \cap BG_x^z \neq \emptyset\} =\emptyset
\end{align*}
(where $h$ is the smallest integer such that $E_{h+1}^z$ is empty), and the set
$$
E^z=\bigcup_{i=1}^h E_i^z.
$$
For the sake of completeness, we define all these sets to be empty if $z$ is not free.
They prove the following theorem.
\begin{teo}\label{teo:universal_basin}
Let $x$ and $z$ be two social outcomes.
Then $x$ is in the universal basin of attraction $\Psi(z)$ if and only if $x$ belongs to $E^z$, i.e.
$$
\Psi(z)=E^z.
$$
\end{teo}

\begin{defi}
Let $z$ be a social outcome.
The {\em deepness} of a social outcome $x$ with respect to $z$ is
\begin{itemize}
\item
$d$ if $x$ belong to $E_d^z$,
\item
$\infty$ if $x$ does not belong to $\Psi(z)$.
\end{itemize}
\end{defi}
Note that this definition makes sense because the $E_*^z$'s form a partition of the universal basin of attraction of $z$.

\begin{prop}
The deepness of a social outcome $x$ with respect to $z$ is the minimum of the lengths of all maximal domination paths $DP(x,z,A_z)$, among all objects schemes $A_z$ such that $\Phi(z,A_z)$ is empty.
\end{prop}

\begin{proof}
Let $d$ be the deepness of the social outcome $x$ with respect to $z$ and let $h$ be the minimum of the lengths of all maximal domination paths $DP(x,z,A_z)$, among all objects schemes $A_z$ such that $\Phi(z,A_z)$ is empty.
If $d$ is $\infty$, by virtue of Theorem~\ref{teo:universal_basin}, there is no maximal domination path $DP(x,z,A_z)$, where $A_z$ is an objects scheme such that $\Phi(z,A_z)$ is empty, and hence $h$ is $\infty$.

If $d$ is not $\infty$, we can construct a maximal domination path
$$
x = x_d \prec x_{d-1} \prec \dotsb \prec x_1 \prec x_0 = z
$$
such that $x_j$ belongs to $E^z_j\cap BG^z_{x_{j+1}}$ for $j=0,\dotsc,d-1$ and hence we have $h\leq d$.
Let $DP(x,z,A_z)$ be a maximal domination path
$$
x = x_h \prec x_{h-1} \prec \dotsb \prec x_1 \prec x_0 = z
$$
of length $h$.
If $\module_{I_j}$ is the object of $A_z$ such that $x_{j-1}$ belongs to $B(x_j,\module_{I_j})$, we have $\hyperplanesep{x_j,x_{j-1}}\subseteq\module_{I_j}$ and $x_{j-1}\in E^z_{j-1} \cap BG^z_{x_j}$.
Thus, $x$ belongs to $E^z$ and hence to some $E^z_k$ with $k\leq h$.
Since the deepness of $x$ is $d$, we have $k=d$, and then
$$
d=k\leq h\leq d.
$$
The proof is complete.
\end{proof}

\begin{defi}
The {\em u-deepness} of a social outcome $z$ is
\begin{itemize}
\item
the maximum integer $h$ such that $E^z_h$ is not empty,
\item
$-\infty$ if all $E^z_h$'s are empty.
\end{itemize}
\end{defi}
Note that the u-deepness of a social outcome $z$ is $-\infty$ if and only if $z$ is not free.

\section{Theoretical results}
\label{sec:results}

In this section we will give results that ties the model described in~\cite{MarengoSette} with tournament theory.
From now on, we will denote by $\calT_i$'s the irreducible components of the graph $\Cal{Y}_{\succ}$, with $\maxcomp{\calT}$ being the maximal component.

\begin{prop}\label{prop:univ_basin_irred_comp}
If a social outcome $x \in \calT_i$ is in the universal basin of attraction $\Psi(z)$ of a social outcome $z \in \calT_j$, then $i \leqslant j$ holds.
\end{prop}

\begin{proof}
Since $x$ belongs to $\Psi(z)$, there exists an objects scheme $A$, an agenda $\alpha$ and a maximal domination path $DP(x,z,A_{\alpha})$,
$$
x = x_0 \prec x_1 \prec \dotsb \prec x_s = z,
$$
ending up in $z$.
Two cases may occur: either $x \succ z$ or $x \prec z$.
In the former case, there exists a domination cycle $\gamma$ that contains $x$ and $z$, i.e.~we have $i=j$.
In the latter case, we have $i \leqslant  j$.
This concludes the proof.
\end{proof}

\begin{cor}\label{cor2}
Each u-local optimum belongs to $\maxcomp{\calT}$.
\end{cor}

\begin{rem}
The converse of the above corollary is not true.
For instance, the social rule whose graph is shown in Figure~\ref{fig:maxcomp_no_globalopt} has only one irreducible component and no u-local optimum.
\begin{figure}[t]
  \begin{center}
        $\xymatrix{
        \nodeone{0} \arc{[r]} & \nodeone{1} \arc{[r]} & \nodeone{2}
          \arccurv{[ll]}{@/_13pt/}
        }$
  \end{center}
  \caption{A social rule with no (u-)local optimum.}
  \label{fig:maxcomp_no_globalopt}
\end{figure}
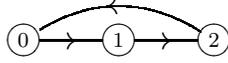

Moreover, for a social outcome $z$ the property of being a local optimum for 
an objects scheme (and agenda) and the property of belonging to $\maxcomp{\calT}$ are not 
related to each other.
The social rule whose graph is shown in Figure~\ref{fig:maxcomp_no_globalopt} has no local optimum, while that shown in Figure~\ref{fig:localopt_no_maxcomp}, for the objects scheme $\{\{H_1\},\{H_2\}\}$, has a local optimum, $00$, which is not in $\maxcomp{\calT}=\{11\}$.
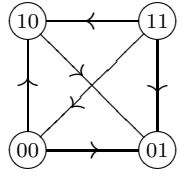
\begin{figure}[t]
  \begin{center}
        $\xymatrix@R=35pt@C=35pt{
        \nodetwo{10} \arccurvshift{[rd]}{}{0.4} & \nodetwo{11} \arc{[d]} \arc{[l]}
          \arccurvshift{[ld]}{}{0.7} \\
        \nodetwo{00} \arc{[r]} \arc{[u]}  & \nodetwo{01}
        }$
  \end{center}
  \caption{A social rule with a local optimum, $00$, which is not in $\maxcomp{\calT}=\{11\}$.}
  \label{fig:localopt_no_maxcomp}
\end{figure}
\end{rem}

\begin{prop}
A social outcome $z \in X$ is a local optimum for all objects schemes if and 
only if $z$ is the only element in $\maxcomp{\calT}$.
\end{prop}

\begin{proof}
If $z \in X$ is a local optimum for all objects schemes, in particular it is a local optimum for $A=\{\Cal A_{n,m}\}$.
Then, we have $B(z,\Cal A_{n,m})=\emptyset$, i.e.~we have $z \succ x$ for all $x \in X\setminus\{z\}$ and hence $z$ is the only social outcome in $\maxcomp{\calT}$.
The converse is obvious.
\end{proof}

\begin{defi}
Let $x \in \calT_i$ be a social outcome.
We say that $x$ is {\em lifting with respect to an objects scheme $A$} if there is an object $\Cal A \in A$ such that the best neighbor $y \in B(x,\Cal A)$ belongs to a component $\calT_j$ such that $j >i$.
\end{defi}

By definition, in $\maxcomp{\calT}$ there are no lifting social outcomes.
Indeed, social outcomes that are lifting with respect the objects scheme $A$ arise when an arc in a domination path through $A$ has the endpoints in two different irreducible components.
In the following theorem we will give an equivalent condition for a
social outcome $x \in X$ to be lifting.

\begin{teo}
A social outcome $x$ in an irreducible component $\calT_i$ is lifting for at least an objects scheme $A$ if and only if there exists a social outcome $y \in \calT_j$ with $j>i$ such that the following condition holds:
$$
\Cal{H}_{w,y} \nsubseteq \Cal{H}_{x,y} \quad
\forall w \in X\ \text{such that}\ w \succ y.
$$
\end{teo}

\begin{proof}
Let $x \in \calT_i$ be a lifting social outcome for an objects scheme $A$.
Then, there exists $y \in \calT_j$ with $j>i$ such that
$B(x,\Cal A)=\{y\}$ for an object $\Cal A \in A$.
By construction, we have $\Cal{H}_{x,y} \subseteq \Cal A$.
Suppose by contradiction that there exists a social outcome $w \succ y$ such that $\Cal{H}_{w,y} \subseteq \Cal{H}_{x,y}$.
We have $\Cal{H}_{w,y} \subseteq \Cal A$, so $y$ cannot belong to $B(x, \Cal A)$, a contradiction.

Conversely, let $y \in \calT_j$, with $j>i$, be a social outcome such that $\Cal{H}_{w,y} \nsubseteq \Cal{H}_{x,y}$ for all $w \succ y$.
Then, we have $y \succ x$.
Moreover, for each social outcome $w \succ y$ we have
$w(\Cal{H}_{x,y}^c) \neq y(\Cal{H}_{x,y}^c)$, i.e.~$w$ is a neighbor neither of $x$ nor of $y$ with respect to $\Cal{H}_{x,y}$.
Therefore, we obtain $B(x, \Cal{H}_{x,y})=\{y\}$ and hence the thesis.
\end{proof}

Until the end of this section, we fix a social outcome $z$, which will be a candidate for being an u-local optimum.
We give the following necessary conditions on the irreducible 
components $\calT_i$ of the graph $\Cal{Y}_{\prec}$ in order for $z$ to be an u-local optimum.
\begin{prop}\label{prop1}
If $z$ is an u-local optimum, the following statements hold.
\begin{enumerate}[(i)]
\item\label{item:prop1_exist}
For each social outcome $x\in\calT_i$, with $\calT_i\neq\maxcomp{\calT}$, there exist an objects scheme $\scheme_x$ such that $\Phi(z,A_x)$ is empty and a domination path $DP(x,y,A_x)$ to a social outcome $y\in\calT_i$ lifting with respect to $A_x$.
\item\label{item:prop1_lifting}
For each social outcome $x\in\calT_i$, with $\calT_i\neq\maxcomp{\calT}$, every domination path $DP(x,z,A)$ through an objects scheme $A$ such that $\Phi(z,A)$ is empty contains a social outcome $y\in\calT_i$ lifting with respect to $A$.
\item\label{item:prop1_contain}
Each $\calT_i$ different from $\maxcomp{\calT}$ contains a lifting social outcome with respect to an objects scheme $A$ such that $\Phi(z,A)$ is empty.
\end{enumerate}
\end{prop}

\begin{proof}
Let us prove Point~(\ref{item:prop1_exist}).
Let $x$ be a social outcome belonging to $\calT_i$, with $\calT_i\neq\maxcomp{\calT}$.
Since $x$ belongs to $\Psi(z)$, there exist an objects scheme $\scheme$ such that $\Phi(z,\scheme)$ is empty and a maximal domination path $DP(x,z,A)$,
$$
x = x_0 \prec x_1 \prec \dotsb \prec x_s = z,
$$
ending up in $z$.
For $j=0,\dotsc,s$, let us define the integer $i_j$ such that $x_j$ belongs to $\calT_{i_j}$.
These integers are ordered non decreasingly and $i_0$ differs from $i_s$.
Therefore, there exists a maximal one (say $\bar{\jmath}$) different from $s$ and such that $i_{\bar{\jmath}}=i_0$ and $i_{\bar{\jmath}+1}\neq i_0$ hold.
The social outcome $y=x_{\bar{\jmath}}$ belongs to $\calT_i$ and is lifting with respect to $\scheme$, so the domination path
$$
x = x_0 \prec x_1 \prec \dotsb \prec x_{\bar{\jmath}} = y,
$$
is the path we are looking for.

A proof of Point~(\ref{item:prop1_lifting}) is very similar to that of Point~(\ref{item:prop1_exist}), so we leave it to the reader.
Point~(\ref{item:prop1_contain}) is a direct consequence of Point~(\ref{item:prop1_exist}).
\end{proof}

\begin{prop}\label{prop:comp_lifting}
Suppose that there is an irreducible component $\calT_i\neq\maxcomp{\calT}$ such that for each 
$x\in\calT_i$ we have $G^z_x\subseteq\calT_i$ (or equivalently $BG_x^z\subseteq\calT_i$).
Then z is not an u-local optimum.
\end{prop}

\begin{proof}
Suppose by contradiction that $z\in\confspace$ is an u-local optimum.
Let $x$ be a social outcome that belongs to $\calT_i$.
By means of Theorem~\ref{teo:universal_basin}, we obtain an objects scheme $A$ such that $\Phi(z,A)$ is empty and a domination path $DP(x,z,A)$,
$$
x = x_0 \prec x_1 \prec \dotsb \prec x_s \prec x_{s+1} = z,
$$
where $x_j$ belongs to $E^z$ for each $j=0,\dotsc,s+1$, i.e.~$x_{j+1}$ belongs to $BG_{x_j}^z$ for each $j=0,\dotsc,s$.
By virtue of Proposition~\ref{prop1}-(\ref{item:prop1_lifting}), there exists $\bar{\jmath}\in\{0,\dotsc,s\}$ such that $x_{\bar{\jmath}}$ belongs to $\calT_i$ and is lifting with respect to $A$, a contradiction to the hypothesis.
\end{proof}

We will denote by
$$
{\cal S}^z_i=\{x\in\calT_i\mid G^z_x\subseteq\calT_i\}
$$
the set of the social outcomes of the irreducible component $\calT_i$ that are not lifting with respect to any objects scheme $\scheme$ such that $\Phi(z,A)$ is empty.
Therefore, the proposition above can be restated as follows.
\addtocounter{teo}{-1}
\begin{prop}
If $z$ is an u-local optimum, then ${\cal S}^z_i\neq\calT_i$ holds for all $\calT_i\neq\maxcomp{\calT}$.
\end{prop}

For each irreducible component $\calT_i$, we will now construct a particular sub-graph of $\calT_i$.
It will give information on the possible domination paths through an objects scheme, starting from a social outcome of $\calT_i$ and ending up in $z$.
The nodes of this graph are the social outcomes in $\calT_i$; if $x$ and $y$ are social outcomes, there is an arc from $x$ to $y$ if $y\in BG_x^z$.
Note that lifting social outcomes are maximal elements of this graph.
\begin{prop}
Suppose $z$ is an u-local optimum.
Then the following two conditions are satisfied:
\begin{itemize}
\item
for each irreducible component $\calT_i\neq\maxcomp{\calT}$, each maximal element of the graph constructed above (considered as a social outcome) is lifting with respect to an objects scheme $\scheme$ such that $\Phi(z,A)$ is empty;
\item
for $\maxcomp{\calT}$, the u-local optimum $z$ is the only maximal element of the graph constructed above.
\end{itemize}
\end{prop}
\begin{proof}
Let $y\in\calT_i\neq\maxcomp{\calT}$ be a maximal element of the graph constructed
above.
By virtue of Proposition~\ref{prop1}-(\ref{item:prop1_exist}) there exists a domination path
$$
y=y_0 \prec y_1 \prec \dotsb \prec y_d = y'
$$
from $y$ to a lifting social outcome $y'\in\calT_i$.
Since $y$ is a maximal element, the set $BG_y^z \cap \calT_i$ is empty and hence $d$ is zero.
Therefore, $y'$ equals $y$ and hence $y$ is lifting.

Similarly, let $y\in\maxcomp{\calT}$ be a maximal element of the graph constructed
above.
The set $BG_y^z \cap \maxcomp{\calT}$ is empty and hence the whole $BG_y^z$ is empty.
Since $y$ belongs to $\Psi(z)$, we obtain that $y$ equals $z$.
This concludes the proof.
\end{proof}

\begin{prop}
The score of a local optimum is at least $\sum_{j=1}^n (m_j-1)$.
\end{prop}
\begin{proof}
By virtue of Theorem~\ref{principio} each local optimum $z$ must dominate the $\sum_{j=1}^n (m_j-1)$ social outcomes $w$ with $d_p(w,z)=1$.
\end{proof}

\begin{rem}
The bound in the proposition above seems to be quite weak, mainly because in the classical social choice framework the score of an optimum is $M-1$.
However, the bound above is attained.
Namely, let $\succ$ be any social rule such that $z\succ x$ for each social outcome $x$ with $d_p(x,z)=1$, and $z\prec x$ for each social outcome $x$ with $d_p(x,z)>1$.
For $\succ$ the social outcome $z$ is free (and hence is a local optimum for an objects scheme $A_z$) and has score $\sum_{j=1}^n (m_j-1)$.

Moreover, there can be also global optima with score $\sum_{j=1}^n (m_j-1)$.
This can be easily obtained by suitably choosing the arcs of the social rule $\succ$ that are not fixed above, so we leave it to the reader.
\end{rem}

\section{Probability}
\label{sec:probability}

As above, let $X$ be the set of possible social outcomes given by a bundle of features $F=\{f_1,\dotsc ,f_n\}$ such that
$f_i$ belongs to $\{0,1,2,\dotsc, m_i-1\}$ for $i=1,\dotsc,n$.
Throughout this section, we will suppose (without loss of generality) that the $m_i$'s are ordered decreasingly:
$$
m_1 \geqslant m_2 \geqslant \dotsb \geqslant m_n.
$$

\begin{teo}\label{teo:max_free}
In the hypothesis above, any given social rule $\succ$ on $X$ has at most
$$
\prod_{i=2}^{n} m_i
$$
local optima, and this bound is attained.
\end{teo}

\begin{proof}
The proof of the bound is by induction on the number $n$ of features.
If $n$ is $1$, there is at most one local optimum (one when the social
rule is transitive, zero otherwise).
Suppose now the statement true for $n$ and suppose $\succ$ is defined on social oucomes with $n+1$ features.
If $j$ belongs to $\{0,1,2,\dotsc, m_{n+1}-1\}$ for the $(n+1)$-th feature, we define the subspace $V_n^j=\{(y_1,\ldots,y_{n+1})\in \R^{n+1} \mid y_{n+1}=j-\frac12\}$ of $\R^{n+1}$ having dimension $n$.
Let $X_n^j$ be the set of all the social outcomes in $X$ whose corresponding chambers intersect $V_n^j$.
Then, for any $j\in\{0,1,2,\dotsc, m_{n+1}-1\}$,  $X_n^j$ is 
given by $n$ features taking $m_1 \geqslant \dotsb \geqslant m_n$ values.
By induction, $X_n^j$ has at most $\prod_{i=2}^n m_i$ local optima for any $j\in\{0,1,2,\dotsc, m_{n+1}-1\}$.
Moreover, by definition of local optimum, if $x \in X$ is a local optimum for the social rule $\succ$, it is a local 
optimum also for the social rule $\succ^j$ induced by $\succ$ on $X_n^j$.
Therefore, the number of local optima in $X$ is at most the sum 
over $j\in\{0,1,2,\dotsc, m_{n+1}-1\}$ of the maximum number of local optima of each $X_n^j$, i.e. $\prod_{i=2}^{n+1}m_i$.
This concludes the proof of the bound.

In order to prove that the bound is attained, we will prove the following slightly stronger statement.

\smallskip
\noindent
\textsc{Assertion.}
In the hypothesis above, there exist at least $m_n$ social rules with exactly $\prod_{i=2}^{n}m_i$ local optima and such that for any two of them the sets of local optima are disjoint.
\smallskip

The proof of the assertion is by induction on the number $n$ of features.
If $n$ is $1$, any transitive social rule has one local optimum (the global one).
Moreover, for any social outcome $z$ there is a transitive social rule with $z$ as local optimum.
Since there are $m_1$ social outcomes, we obtain the thesis.

Suppose now the statement true for $n$ and let $\confspace$ be a set of social outcomes with $n+1$ features.
As above, we define the subspace $V_n^j$ and the set $X_n^j$, for $j=0,\dotsc, m_{n+1}-1$.
By induction, on each $X_n^j$ we can define a social rule $\succ^j$ with exactly $\prod_{i=2}^{n}m_i$ local optima and such that the local optima in two different $X_n^j$ are separated by at least two features (one of which being $f_{n+1}$), because $m_{n+1} \leqslant m_n$ holds.
More precisely, if $v_1\dotsm v_n j \in X_n^j$ is a local optimum for $\succ^j$ and $v'_1\dotsm v'_n j' \in X_n^{j'}$ is a local optimum for $\succ^{j'}$ with $j\neq j'$, then there exists a feature $f_k$ different from $f_{n+1}$ such that $v_k$ differs from $v'_k$.
Therefore, there exists a social rule $\succ$ on $\confspace$ that satisfies the following properties:
\begin{itemize}
\item $\succ$ equals $\succ^j$ on $X_n^j$,
\item if $x \in X_n^j$ is a local optimum for $\succ^j$ then $x\succ y$ for all $y \in X \setminus X_n^j$ such that $d_p(x,y)=1$.
\end{itemize}
This social rule has $\prod_{i=2}^{n+1}m_i$ free social outcomes, which are local optima by virtue of Theorem~\ref{principio}; therefore, it is one of the social rules we are looking for.

The other ones can be obtained by shifting the pairing of $X_n^*$'s and $\succ^*$'s.
More precisely, for $l=0,\dotsc,m_{n+1}-1$, the $l$-th social rule $\succ_l$ is defined by choosing on $X_n^j$ the social rule $\succ^h$, where $h$ is the remainder of the division of $j+l$ by $n+1$, and by repeating the procedure above.
These $m_{n+1}$ social rules on $\confspace$ have $\prod_{i=2}^{n+1}m_i$ local optima each.
Moreover, if $v_1\dotsm v_n j \in X_n^j$ is a local optimum for $\succ_l$ and $v'_1\dotsm v'_n j' \in X_n^{j'}$ is a local optimum for $\succ_{l'}$ with $l\neq l'$, then by construction either $j$ differs from $j'$ or $v_1\dotsm v_n$ differs from $v'_1\dotsm v'_n$.
This concludes the proof of the assertion, and hence the proof of the theorem.
\end{proof}

\paragraph{Social rules with a fixed number of free social outcomes in the two-feature case}
We compute the number of social rules with two features and a fixed number of free social outcomes.
Note that, by virtue of Theorem~\ref{teo:max_free}, there are at most $m_2$ free social outcomes.
Let us call $e_k$ the number of social rules with $k$ free social outcomes.
We will count the graphs corresponding to the social rules.

Call $V_1$ (resp.~$V_2$) the set of values of the first (resp.~second) feature of the $k$ free social outcomes.
Since two free social outcomes are separated by both features, we have $\#V_1=\#V_2=k$.
There are $\binom{m_1}k$  (resp.~$\binom{m_2}k$) possibilities for choosing $V_1$ (resp.~$V_2$).
Moreover, in $V_1\times V_2$, the $k$ free social outcomes can be chosen in $k!$ ways.

Suppose now that the position of the $k$ free social outcomes is fixed.
For each $k=0,1,\dotsc,m_2$, we compute an integer $a_k$ which is related to (but different from) $e_k$, because we allow some repetitions in the counting process.
Since the $k$ social outcomes are free, each of them dominates all the social outcomes that are separated from it by one feature.
Therefore, $k\left(m_1+m_2-2\right)$ arcs are fixed.
If $k$ equals $m_2$, the other $\binom{m_1m_2}2 - m_2\left(m_1+m_2-2\right)$ arcs are unrestricted and hence we obtain
$$
e_{m_2} = \binom{m_1}{m_2} \binom{m_2}{m_2} m_2!  2^{\binom{m_1m_2}2 - m_2\left(m_1+m_2-2\right)}
$$
graphs.
For general $k$, the other $\binom{m_1m_2}2 - k\left(m_1+m_2-2\right)$ arcs are not unrestricted, because there should not be any other free social outcome.
However, if we leave them unrestricted, we obtain
$$
a_k = \binom{m_1}k \binom{m_2}k k! 2^{\binom{m_1m_2}2 - k\left(m_1+m_2-2\right)}
$$
graphs.
In this process we count each graph with $k+l$ free social outcomes $\binom{k+l}{k}$ times and hence we obtain the system of linear equations
$$
\left\{
\begin{array}{l}
\sum_{k=0}^{m_2}\binom{k}{0}e_k = a_0 \\
\sum_{k=1}^{m_2}\binom{k}{1}e_k = a_1 \\
\quad\quad\quad\vdots \\
e_{m_2-1} + m_2e_{m_2} = a_{m_2-1} \\
e_{m_2}=a_{m_2}.
\end{array}
\right.
$$
By (partially) solving this system, we obtain the recursive formula
$$
e_k = a_k - \sum_{l=k+1}^{m_2}\binom{l}{k}e_l
$$
for computing the number of social rules with two features and $k$ free social outcomes.
An explicit formula can be given.
If $S$ is a subset of $\{k,k+1,\dotsc,i\}$, we denote by $\mathop{\mathrm{Prod}}(S)$ the product
$$
\prod_{j=1}^{\#S-1} \binom{s_{j+1}}{s_j},
$$
where the $s_*$'s are the elements of $S$ ordered increasingly ($s_1<s_2<\dotsb<s_{\#S}$).
The number of social rules with two features and $k$ free social outcomes is 
$$
e_k =
\sum_{i=k}^{m_2}
\left(
\sum_{\substack{S\subseteq\{k,k+1,\dotsc,i\}\\ k,i\in S}}
(-1)^{\#S+1} \mathop{\mathrm{Prod}}(S)
\right)
a_i.
$$
The proof of this formula, by means of a recursion from $m_2$ to zero, is straightforward, so we leave it to the reader.

With an effort one may carry out a similar argument to compute the number of social rules with three features and a fixed number of free social outcomes, but a general formula seems to be unfeasible with this technique.

An interesting issue is to study the probability $P_{(m_1,\dotsc,m_n)}(k)$ that a social rule has $k$ free social outcomes.
For the social rules with two features, this quotient is
$$
P_{(m_1,m_2)}(k) = \frac{e_k}{2^{\binom{m_1m_2}2}}.
$$
In Table~\ref{table:prob_fixed_outcomes} we have computed it for small values of $m_1$ and $m_2$.
\begin{table}
  \begin{center}
  \begin{small}
  \begin{tabular}{c@{\hspace{24pt}}cccc}
  \toprule
  free social & \multicolumn{4}{c}{values for the features} \\
  outcomes & (2,2) & (3,3) & (5,5) & (10,10) \\
  \midrule
  0 & .125 & .5063476563 & .9053598846 & .9996185892 \\
  1 & .75 & .4262695313 & .0916594645 & .0003813519 \\
  2 & .125 & .0659179688 & .0029453066 & .0000000589 \\
  3 & . & .0014648438 & .0000352051 & $<10^{-10}$ \\
  4 & . & . & .0000001392 & $<10^{-10}$ \\
  5 & . & . & .0000000001 & $<10^{-10}$ \\
  6 & . & . & . & $<10^{-10}$ \\
  7 & . & . & . & $<10^{-10}$ \\
  8 & . & . & . & $<10^{-10}$ \\
  9 & . & . & . & $<10^{-10}$ \\
  10 & . & . & . & $<10^{-10}$ \\
  \bottomrule
  \end{tabular}
  \end{small}
  \end{center}
  \caption{The probability that a social rule with two features has a fixed number of free social outcomes.}
  \label{table:prob_fixed_outcomes}
\end{table}

\paragraph{Decidability and manipulability in the new framework}
In the classical social choice framework a given social outcome $z$ is an optimum if and only if it dominates all the other social outcomes.
Therefore, the probability $P(z)$ that a given social outcome $z$ is an optimum for a social rule on $M$ social outcomes is given by the quotient between the number of graphs with $M-1$ nodes and the number of graphs with $M$ nodes, i.e.
$$
P(z)=\frac{2^{\binom{M-1}{2}}}{2^{\binom{M}{2}}} = \frac{1}{2^{M-1}}.
$$

In Marengo and the second author's model, global optima play the role of optima in the classical
framework, but also a local optimum can be an optimum if the agents vote starting from a particular social outcome.
The probability $P(z)$ for a given social outcome $z$ to be a local optimum is given by
the quotient between the number of the graphs with $M$ nodes and with
$\sum_{i=1}^{n}m_i -n $ fixed arcs, and the number of
all the graphs with $M$ nodes, i.e.
$$
P(z)=\frac{2^{\binom{M}{2}-(\sum_{i=1}^{n}m_i-n)}}{2^{\binom{M}{2}}}
= \frac{1}{2^{\sum_{i=1}^{n}m_i -n}}=\frac{2^n}{2^{\sum_{i=1}^{n}m_i}}.
$$

It is clear that, if $n$ is greater than $1$, the probability for $z$ to be a local optimum is far greater than that to be an optimum in the classical framework.
Therefore, we define a function $F: \mathbb N^3 \longrightarrow \mathbb
Q $ depending on $n$, $M=\prod_{i=1}^n m_i$ and $\sigma=\sum_{i=1}^n
m_i$, defined to be the quotient between the probability of a social outcome to be an optimum
in the classical framework and that to be a local optimum in the new model,
$$
F(n,M,\sigma)=\frac{2^n}{2^{\sum_{i=1}^{n}m_i}}2^{M-1}=2^{n+M-(\sigma-1)}.
$$
\begin{rem}
The inequality
$$
F(n,M,\sigma) \geqslant 1
$$
holds.
However, the inequality becomes strict,
$$
F(n,M,\sigma) > 1,
$$
if and only if $n$ is greater than $1$.
\end{rem}

The study of the function $F$ seems to be very important in the social choice context.
Indeed, it gives an idea about the decidability and the manipulability of choice in the new model with respect to the old one.
We note that, for example, if $m_i$ is $2$ for $i=1,\dotsc,n$, there is a high probability for a social outcome to be a local optimum, while this probability strongly decreases when the values of the $m_i$'s increase.
However, the value of $F$ is far greater than $1$ even if the $m_i$'s are greater than $2$.
We think that this function $F$ is a measure of the power of this new approach in social decision theory.

\section{The algorithm}
\label{sec:algorithm}

We will describe here an algorithm to compute the universal basin of attraction of a social outcome $z$ of a social rule $\prec$.
It finds also the sets $E^z_i$ defined at the end of Section~\ref{sect:model}.
Therefore it obtains also the deepness of each social outcome with respect to $z$ and the u-deepness of $z$.

The algorithm \textsc{ComputeUniversalBasin} works as follows.
The pseudocode is shown in Algorithm~\ref{alg:ComputeUniversalBasin}.
\begin{algorithm}
\begin{footnotesize}
\begin{algorithmic}[1]
\REQUIRE A social rule $\succ$ and a social outcome $z$
\ENSURE The non-void $E^z_i$'s and hence the universal basin of attraction $\Psi(z)=\bigcup_i E^z_i$.
\STATE \label{line:initialise_start} Initialize $\confspace \duepuntiuguale$ the set of all social outcomes
\FOR {$i=0$ \textbf{to} $M$}
	\STATE Initialize $E^z_i \duepuntiuguale \emptyset$
\ENDFOR \label{line:initialise_end}
\STATE \label{line:compute_irred} Compute the irreducible components $\calT_j$
\STATE \label{line:after_irred} Let $h \duepuntiuguale$ the integer such that $z\in\calT_h$
\STATE \label{line:remove_irred} $X \duepuntiuguale X\setminus\bigcup_{j>h}\calT_j$
\IF[$z$ is not a local optimum] {$\exists w\succ z$ such that $\#\featuresep{w,z}=1$} \label{line:is_local_opt}
	\RETURN \label{line:empty_basin} $\emptyset$ \COMMENT{$E^z_j=\emptyset\ \forall j=0,\dotsc,M$}
\ELSE[$z$ is a local optimum]
	\STATE \label{line:n1_start} $E^z_0 \duepuntiuguale \{z\}$
	\STATE $\confspace \duepuntiuguale \confspace\setminus\{z\}$
	\STATE \label{line:Bz} Compute $B(z) \duepuntiuguale \{y\in\confspace |
		y\prec z,\
		\featuresep{z,y}\not\supseteq\featuresep{w,z}\ \forall w\succ z\}$
	\STATE $E^z_1 \duepuntiuguale B(z)$
	\STATE \label{line:n1_end} $\confspace \duepuntiuguale \confspace\setminus B(z)$
	\RETURN \textsc{Recursion}($1$)
\ENDIF \label{line:first_part_end}
\STATE
\STATE \label{line:recursion} \textbf{where} \textsc{Recursion} \textbf{is the following function}
\REQUIRE An integer $i$
\ENSURE The non-void $E^z_i$'s and hence the universal basin of attraction $\Psi(z)=\bigcup_i E^z_i$.
\FORALL {$y_i\in E^z_i$} \label{line:moving_new_elm_start}
	\STATE \label{line:Byi} Compute $B(y_i) \duepuntiuguale \{y_{i+1}\in\confspace |
		y_{i+1}\prec y_i,\
		\featuresep{y_i,y_{i+1}}\not\supseteq\featuresep{w,y_i}\ \forall w\succ y_i,$\\
		$\featuresep{y_i,y_{i+1}}\not\supseteq\featuresep{w,z}\ \forall w\succ z\}$
	\STATE \label{line:Ei1_Byi} $E^z_{i+1} \duepuntiuguale E^z_{i+1}\cup B(y_i)$
	\STATE \label{line:X_Byi} $\confspace \duepuntiuguale \confspace\setminus B(y_i)$
\ENDFOR \label{line:moving_new_elm_end}
\IF[the universal basin of attraction has been computed] {$E^z_{i+1}=\emptyset$} \label{line:another_step}
	\RETURN \label{line:rec_stop} $E^z_j$ for all $j=0,\dotsc,i$ \COMMENT{$E^z_j=\emptyset\ \forall j>i$}
\ELSE[the next step]
	\RETURN \label{line:rec_recall} \textsc{Recursion}($i+1$)
\ENDIF
\end{algorithmic}
\end{footnotesize}
\caption{Pseudocode for the algorithm to compute the universal basin of attraction of a social outcome $z$ of a social rule $\prec$. (All sets are implemented as ordered lists.)}
\label{alg:ComputeUniversalBasin}
\end{algorithm}

\begin{enumerate}[Step~1.]
\item
Consider the set $\confspace$ of all social outcomes.
Start with the empty sets $E^z_i$ for $i\in\mathbb{N}$ (we will need only a finite number of them).
Eventually, the universal basin of attraction $\Psi(z)$ will be $\bigcup_i E^z_i$.

\item
\label{step:remove_irred}
Compute the irreducible components $\calT_*$ of the graph corresponding to the social rule $\prec$.
Let $h$ be the integer such that $z\in\calT_h$.
Remove from $X$ all the social outcomes that are in $\calT_j$ for all $j>h$.

\item
If $\featuresep{w,z}$ is made up of only one feature for some $w\succ z$, then $z$ is not a local optimum and hence $E^z_i$ is empty for all $i\in\mathbb{N}$: go to Step~\ref{item:final_step}.
Otherwise, $z$ is a local optimum: hence add $z$ to $E^z_0$ and remove it from $X$.

\item
Find all the social outcomes $y\in\confspace$ such that
\begin{gather*}
y\prec z,\\
\featuresep{z,y}\not\supseteq\featuresep{w,z}\ \mathrm{for\ all}\ w\succ z.
\end{gather*}
Add these $y$'s to $E^z_1$ and remove them from $\confspace$.
Go to Step~\ref{step:iter} with $i=1$.

\item
\label{step:iter}
For each $y_i\in E^z_i$ do the following steps.
\begin{itemize}
\item
Find all the social outcomes $y_{i+1}\in\confspace$ such that
\begin{gather*}
y_{i+1}\prec y_i,\\
\featuresep{y_i,y_{i+1}}\not\supseteq\featuresep{w,y_i}\
\mathrm{for\ all}\ w\succ y_i,\\
\featuresep{y_i,y_{i+1}}\not\supseteq\featuresep{w,z}\
\mathrm{for\ all}\ w\succ z.
\end{gather*}
\item
Add these $y_{i+1}$'s to $E^z_{i+1}$ and remove them from $\confspace$.
\item
If $E^z_{i+1}$ is empty, go to Step~\ref{item:final_step};
otherwise, repeat Step~\ref{step:iter} with $i$ incremented by $1$.
\end{itemize}

\item
\label{item:final_step}
The universal basin of attraction $\Psi(z)$ is the union of the $E^z_i$'s (only a finite number of them being non-empty).

\end{enumerate}

\begin{teo}
If the social rule $\prec$ is defined on $M$ social outcomes and a social outcome $z$ is given, the algorithm \textsc{ComputeUniversalBasin} computes the universal basin of attraction of $z$ in $O(M^3\log M)$ time.
\end{teo}

\begin{proof}
We start by proving that the algorithm comes to an end.
The algorithm is tail-recursive, hence we need to prove that the recursive function \textsc{Recursion} (Line~\ref{line:recursion}) does not give rise to an infinite loop.
Each time \textsc{Recursion} is called, it may move elements (those belonging to $B(y_i)$ for each $y_i\in E^z_i$) from $\confspace$ to $E^z_{i+1}$ (Lines~\ref{line:moving_new_elm_start}-\ref{line:moving_new_elm_end}), and then it either stops \textsc{ComputeUniversalBasin} (Line~\ref{line:rec_stop}) or calls itself with $i$ incremented by $1$ (Line~\ref{line:rec_recall}).
Since $\confspace$ is finite, there is a minimal $\bar{\imath}\in\mathbb{N}$ such that all $B(y_{\bar{\imath}})$'s (and hence $E^z_{\bar{\imath}+1}$) are empty.
When \textsc{Recursion}($\bar{\imath}$) is called, it moves no element from $\confspace$ to $E^z_{\bar{\imath}+1}$, and then it stops \textsc{ComputeUniversalBasin} (Line~\ref{line:rec_stop}).

We now prove that the algorithm is correct.
By virtue of Theorem~\ref{teo:universal_basin}, we have that the universal basin of attraction $\Psi(z)$ is $\bigcup_i E^z_i$.
By means of Proposition~\ref{prop:univ_basin_irred_comp}, we know that all the social outcomes that are in $\calT_j$ for all $j>h$ cannot be in the universal basin of attraction of $z$, and hence they can be removed from the set $\confspace$ of social outcomes that may be in the universal basin of attraction (Lines~\ref{line:compute_irred}-\ref{line:remove_irred}).
If $z$ is not a local optimum, all $E^z_i$'s are empty; in this case the condition in Line~\ref{line:is_local_opt} is true (in virtue of Theorem~\ref{principio}) and hence the output is the empty set (Line~\ref{line:empty_basin}).
Otherwise, if $z$ is a local optimum, $E^z_0$ is $\{z\}$ and $E^z_1$ is not empty (see the end of Section~\ref{sect:model}); in this case $E^z_0$ and $E^z_1$ are computed, and their elements are removed from $\confspace$ (Lines~\ref{line:n1_start}-\ref{line:n1_end}).
Then \textsc{Recursion} is called and $E^z_2$ is computed by means of the conditions of the end of Section~\ref{sect:model} (Lines~\ref{line:moving_new_elm_start}-\ref{line:moving_new_elm_end}).
If $E^z_2$ is empty (Line~\ref{line:another_step}), the universal basin of attraction has been computed (see the end of Section~\ref{sect:model}) and the output is $E^z_0$, $E^z_1$ and $E^z_2$ (Line~\ref{line:rec_stop}).
Otherwise, \textsc{Recursion} is called again and $E^z_3$ is computed.
An easy recursive argument now concludes the proof of the correctness of the algorithm.

We eventually show that our algorithm has run-time $O(M^3\log M)$.
First of all, we note that if $x$ and $y$ are social outcomes, $\featuresep{x,y}$ has $O(\log M)$ elements at most and hence it can be computed in $O(\log M)$ time.
Moreover, if $\bar{x}$ and $\bar{y}$ are other social outcomes, the test $\featuresep{x,y}\not\supseteq\featuresep{\bar{x},\bar{y}}$ can also be performed in $O(\log M)$ time.
We denote by $M_i$ the cardinality of the set $E^z_i$ for $i=0,\dotsc,M$.

The initializations in Lines~\ref{line:initialise_start}-\ref{line:initialise_end} are done in $O(M)$ time.
The computation of the irreducible components of Line~\ref{line:compute_irred} is done in $O(M^2)$ time (see Section~\ref{sec:graphs_tournaments}).
Lines~\ref{line:after_irred} and~\ref{line:remove_irred} are executed in $O(M)$ time.
The condition in Line~\ref{line:is_local_opt} can be checked in $O(M\log M)$ time.
Line~\ref{line:Bz} is executed in $O(M^2\log M)$ time and hence the same holds for Lines~\ref{line:n1_start}-\ref{line:n1_end}.

We will now take the call of \textsc{Recursion}($i$) into account.
Line~\ref{line:Byi} is executed in $O(M^2\log M)$ time, while Lines~\ref{line:Ei1_Byi} and~\ref{line:X_Byi} are executed in $O(M)$ time.
These three lines are executed for all $y_i\in E^z_i$, i.e.~$M_i$ times.
The condition in Line~\ref{line:another_step} is checked in $O(1)$ time, and Line~\ref{line:rec_stop} is executed in $O(M)$ time.
Therefore, the call of \textsc{Recursion}($i$) has run-time $O(M_iM^2\log M)$.
Summing the run-time $O(M^2\log M)$ of Lines~\ref{line:initialise_start}-\ref{line:first_part_end} and the run-times of all calls of \textsc{Recursion}, we obtain
$$
M^2\log M + \sum_{i=1}^M M_iM^2\log M =
\left(\sum_{i=0}^M M_i\right)M^2\log M \sim
M^3\log M.
$$
This concludes the proof.
\end{proof}

\begin{rem}
The calculus of irreducible components (Step~\ref{step:remove_irred}) is not necessary, but it can make the computation faster if there are many social outcomes in the irreducible components that dominate $z$.
\end{rem}

\begin{rem}
In Step~\ref{step:iter} the social outcomes with deepness $i+1$ with respect to $z$ are found.
Therefore, the number of calls of the recursive function \textsc{Recursion} is the u-deepness of $z$.
\end{rem}

\begin{rem}
A (faster) simplification of the algorithm described above can be easily constructed to check whether a social outcome is in the universal basin of attraction of another one.
\end{rem}

\paragraph{FOSoR}
The first author has used Algorithm~\ref{alg:ComputeUniversalBasin} to write the computer program {\tt FOSoR}~\cite{FOSoR}.
It reads a social rule and can
\begin{itemize}
\item
compute the universal basin of attractions,
\item
check whether a social outcome is a local (or an u-local) optimum,
\item
check whether a social outcome is in the universal basin of attraction of
another one,
\item
check whether there is a local (or an u-local) optimum,
\item
find the number of local (or u-local) optima,
\item
find an objects scheme (if there is any) through which there is a maximal dominating path starting from a social outcome and ending up in another one,
\item
find the deepnesses and the u-deepnesses.
\end{itemize}

\section{Numerical examples}
\label{sec:numerical}

We give here some numerical results on the numbers of local and u-local optima of social rules.
In order to compute these results the first author has written the computer program {\tt FOSoRStat}~\cite{FOSoRStat}, which is based on Algorithm~\ref{alg:ComputeUniversalBasin}.
It reads the number of values of each feature and the number of random social rules to check.
It works as follows:
\begin{itemize}

\item
it repeatedly
\begin{itemize}
\item creates a random social rule,
\item computes the number of local (and u-local) optima;
\end{itemize}

\item
it computes the percentages and collects the results.

\end{itemize}

We have shown the results for local (resp.~u-local) optima in the case when each feature can assume two values in Table~\ref{table:statistics_local_2} (resp.~Table~\ref{table:statistics_ulocal_2}).
This case is interesting because it represents the binary choice (i.e.~yes/no, true/false, for/against features).
We have shown the results for local (resp.~u-local) optima in some other cases in Table~\ref{table:statistics_local_other} (resp.~Table~\ref{table:statistics_ulocal_other}).
Note that the relative frequencies in the cases with two features are consistent with the probabilities (computed in Section~\ref{sec:probability}) that a social rule with two features has a fixed number of free social outcomes, see Table~\ref{table:prob_fixed_outcomes}.
\begin{table}
  \begin{center}
  \begin{small}
  \begin{tabular}{c@{\hspace{15pt}}c@{\hspace{10pt}}c@{\hspace{10pt}}c@{\hspace{10pt}}c@{\hspace{10pt}}c@{\hspace{10pt}}c@{\hspace{10pt}}c@{\hspace{10pt}}c}
  \toprule
  local & \multicolumn{8}{c}{number of features} \\
  optima & 1 & 2 & 3 & 4 & 5 & 6 & 7 & 8 \\
  \midrule
  0 & . & .125298 & .234797 & .296109 & .328291 & .346168 & .359183 & .363905 \\
  1 & 1 & .749722 & .544492 & .451488 & .410650 & .390143 & .380017 & .375250 \\
  2 & . & .124980 & .206551 & .210998 & .201968 & .194522 & .188194 & .185849 \\
  3 & . & . & .013679 & .038372 & .050962 & .056757 & .058031 & .058879 \\
  4 & . & . & .000481 & .002934 & .007427  & .010837 & .012366 & .013346 \\
  5 & . & . & . & .000097 & .000657 & .001444 & .001964 & .002385 \\
  6 & . & . & . & .000002 & .000043 & .000120 & .000220 & .000341 \\
  7 & . & . & . & . & .000002 & .000009 & .000024 & .000041 \\
  8 & . & . & . & . & . & . & .000001 & .000003 \\
  9 & . & . & . & . & . & . & . & .000001 \\
  \bottomrule
  \end{tabular}
  \end{small}
  \end{center}
  \caption{Relative frequencies of social rules with a fixed number of local optima (over $10^6$ social rules): cases of two values for each feature.}
  \label{table:statistics_local_2}
\end{table}
\begin{table}
  \begin{center}
  \begin{small}
  \begin{tabular}{c@{\hspace{15pt}}c@{\hspace{10pt}}c@{\hspace{10pt}}c@{\hspace{10pt}}c@{\hspace{10pt}}c@{\hspace{10pt}}c@{\hspace{10pt}}c@{\hspace{10pt}}c}
  \toprule
  u-local & \multicolumn{8}{c}{number of features} \\
  optima & 1 & 2 & 3 & 4 & 5 & 6 & 7 & 8 \\
  \midrule
  0 & . &  & .270704 & .353896 & .377606 & .377254 & .374337 & .369871 \\
  1 & 1 &  & .716457 & .608455 & .551384 & .511074 & .472559 & .438681 \\
  2 & . &  & .012335 & .036169 & .066967 & .101588 & .133080 & .157869 \\
  3 & . &  & .000504 & .001460 & .003919 & .009610 & .018503 & .029815 \\
  4 & . &  & . & .000020 & .000123 & .000456 & .001444 & .003468 \\
  5 & . &  & . & . & .000001 & .000016 & .000074 & .000279 \\
  6 & . &  & . & . & . & .000002 & .000003 & .000017 \\
  \bottomrule
  \end{tabular}
  \end{small}
  \end{center}
  \caption{Relative frequencies of social rules with a fixed number of u-local optima (over $10^6$ social rules): cases of two values for each feature.}
  \label{table:statistics_ulocal_2}
\end{table}
\begin{table}
  \begin{center}
  \begin{small}
  \begin{tabular}{c@{\hspace{24pt}}ccccc}
  \toprule
  local & \multicolumn{5}{c}{values for the features} \\
  optima & (3,3) & (3,3,3) & (3,3,3,3) & (5,5) & (10,10) \\
  \midrule
  0 & .5065899 & .6392066 & .7246560 & .905331876 & .9996083 \\
  1 & .4260296 & .3042338 & .2376727 & .091717916 & .0003917 \\
  2 & .0659261 & .0522738 & .0345455 & .002915423 & . \\
  3 & .0014544 & .0041184 & .0029567 & .000034649 & . \\
  4 & . & .0001637 & .0001618 & .000000136 & . \\
  5 & . & .0000037 & .0000071 & . & . \\
  6 & . & . & .0000002 & . & . \\
  \midrule
  repetitions & $10^7$ & $10^7$ & $10^7$ & $10^9$ & $10^7$ \\
  \bottomrule
  \end{tabular}
  \end{small}
  \end{center}
  \caption{Relative frequencies of social rules with a fixed number of local optima (the number of repetitions is indicated in the last line): other cases.}
  \label{table:statistics_local_other}
\end{table}
\begin{table}
  \begin{center}
  \begin{small}
  \begin{tabular}{c@{\hspace{24pt}}cccccccc}
  \toprule
  u-local & \multicolumn{5}{c}{values for the features} \\
  optima & (3,3) & (3,3,3) & (3,3,3,3) & (5,5) & (10,10) \\
  \midrule
  0 & .8020871 & .9699638 & .9966669 & .999923702 & 1 \\
  1 & .1979129 & .0300356 & .0033330 & .000076298 & . \\
  2 & . & .0000006 & . & . & . \\
  \midrule
  repetitions & $10^7$ & $10^7$ & $10^7$ & $10^9$ & $10^7$ \\
  \bottomrule
  \end{tabular}
  \end{small}
  \end{center}
  \caption{Relative frequencies of social rules with a fixed number of u-local optima (the number of repetitions is indicated in the last line): other cases.}
  \label{table:statistics_ulocal_other}
\end{table}

Eventually, we compare Marengo and the second author's model with the classical one.
Note that in the classical model there can be only one optimum and that the probability $P(M)$ that a social rule with $M$ social outcomes has an optimum equals $M$ times the probability that a given social outcome is an optimum, i.e.~$\frac{M}{2^{M-1}}$.
In Table~\ref{table:statistics_old_model} we have computed this probability for small values of $M$.
\begin{table}
  \begin{center}
  \begin{small}
  \begin{tabular}{r@{\extracolsep{12pt}}lcr@{\extracolsep{12pt}}lcr@{\extracolsep{12pt}}lcr@{\extracolsep{12pt}}l}
  \toprule
  $M$ &   $P(M)$ & & $M$ &   $P(M)$ & & $M$ &   $P(M)$ & & $M$ &   $P(M)$ \\
  \cmidrule{1-2}\cmidrule{4-5}\cmidrule{7-8}\cmidrule{10-11}
  2 & 1 & & 6 & .1875 & & 10 & .019531 & & 14 & .001709 \\
  3 & .75 & & 7 & .109375 & & 11 & .010742 & & 15 & .000915 \\
  4 & .5 & & 8 & .0625 & & 12 & .005859 & & 16 & .000488 \\
  5 & .3125 & & 9 & .035156 & & 13 & .003174 & & 17 & .000259 \\
  \bottomrule
  \end{tabular}
  \end{small}
  \end{center}
  \caption{The probability that a social rule has an optimum in the classical model.}
  \label{table:statistics_old_model}
\end{table}

\section{An example}

In this section we will describe in detail an example in which all kinds of optima appear.
It is so small that we can deal with it by hands.
The number of features is three, assuming two values each.
The set $X$ is made up of eight social outcomes: $v_1v_2v_3$ with $v_*=0,1$.
The social rule is any $\succ$ with
\begin{align*}
& 000 \succ 100,\quad 000 \succ 010,\quad 000 \succ 001,\quad 000 \succ 101,\quad 000 \succ 011, \\
& 110 \succ 000,\quad 110 \succ 100,\quad 110 \succ 010,\quad 110 \succ 101,\quad 110 \succ 011, \\
& 101 \succ 100,\quad 101 \succ 001,\quad 101 \succ 111, \\
& 011 \succ 010,\quad 011 \succ 001,\quad 011 \succ 101,\quad 011 \succ 111, \\
& 111 \succ 110,
\end{align*}
where the ten preferences that are not defined are arbitrary.
The preferences are shown in Figure~\ref{fig:example}, where we have disposed the social outcomes as the vertices of a cube.
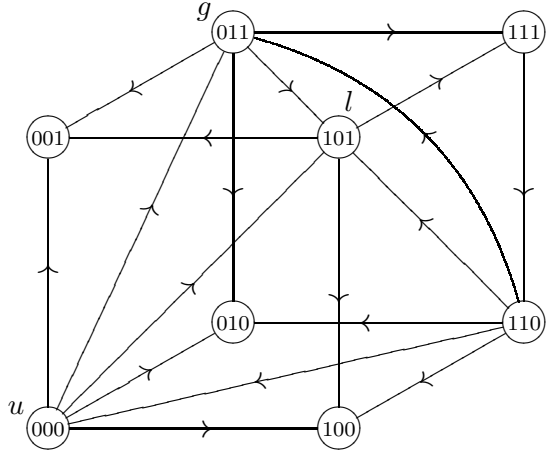
\begin{figure}[t]
  \begin{center}
        $\xymatrix{
        & & \nodethree{011} \ar @{} [l]_<<*+<3pt>{\txt{\begin{small}$g$\end{small}}}
          \arc{[ddd]} \arc{[lld]} \arc{[rd]} \arc{[rrr]}
            & & & \nodethree{111} \arc{[ddd]} \\
        \nodethree{001} & & &
            \nodethree{101} \ar @{} [u]_<<<*-<2pt>{\txt{\begin{small}$l$\end{small}}}
              \arc{[ddd]} \arccurvshift{[lll]}{}{.46} \arc{[rru]} & & \\
        & & & & & \\
        & & \nodethree{010} & & &
            \nodethree{110} \arc{[llllld]} \arc{[lld]} \arc{[lll]} \arc{[lluu]} \arccurvshift{[llluuu]}{@/_24pt/}{.5} \\
        \nodethree{000} \ar @{} [u]^<*+<12pt>{\txt{\begin{small}$u$\end{small}}}
          \arc{[rrr]} \arc{[rru]} \arc{[uuu]} \arccurvshift{[rrruuu]}{}{.5} \arc{[rruuuu]}
            & & & \nodethree{100} & & 
        }$
  \end{center}
  \caption{A social rule with a global optimum, an u-local optimum and a local optimum. (The ten preferences that are not drawn are arbitrary.)}
  \label{fig:example}
\end{figure}

We will show that the social outcome $g=011$ is a global optimum, that the social outcome $u=000$ is an u-local optimum but not a global optimum, and that the social outcome $l=101$ is a local optimum but not an u-local optimum.

The proof that $g$ is a global optimum for the objects scheme $A_g=\{\{H_{1,0},H_{2,0}\},\{H_{3,0}\}\}$ is straightforward, so we leave it to the reader.

In order to prove that $u$ is an u-local optimum, we note that we have
$$
\Psi (u,(\{H_{2,0},H_{3,0}\},\{H_{1,0}\},\{H_{3,0}\})) = \confspace\setminus\{l\}
$$
and
$$
l \in \Psi (u,(\{H_{1,0},H_{3,0}\},\{H_{2,0}\})).
$$
Therefore, $\Psi(u)$ is the whole $\confspace$, and $u$ is an u-local optimum.
We will now prove that $u$ is not a global optimum.
Suppose by way of contradiction that $u$ is a global optimum for an agenda $\alpha$ of an objects scheme $\scheme$, i.e.~$\Psi(u,(\module_1,\dotsc,\module_k))=X$ (where $\module_i=\module_j$ is allowed).
Since $u$ is a local optimum for $\scheme$ (see Remark~\ref{rem:basin_local}), the objects $\{H_{1,0},H_{2,0}\}$ and $\{H_{1,0},H_{2,0},H_{3,0}\}$ cannot belong to $\scheme$.
Since $l$ (resp.~$g$) belongs to $\Psi(u,(\module_1,\dotsc,\module_k))$, the object $\{H_{1,0},H_{3,0}\}$ (resp.~$\{H_{2,0},H_{3,0}\}$) belongs to $\scheme$ and hence $\{H_{1,0},H_{3,0}\}=\module_i$ and $\{H_{2,0},H_{3,0}\}=\module_j$ for some $i$ and $j$ in $\{1,\dotsc,k\}$.
Since $l$ (resp.~$g$) belongs to $\Psi(u,(\module_1,\dotsc,\module_k))$, we have $i<j$ (resp.~$j<i$).
This is a contradiction and hence $u$ is not a global optimum.

Since $l$ is free, it is a local optimum for some objects scheme.
We will now prove that $l$ is not an u-local optimum (then, by virtue of Remark~\ref{rem:glob_uloc_loc}, it is not a global optimum, either).
If $\scheme$ is an objects scheme such that $\Phi(l,A)$ is empty, the objects $\{H_{1,0},H_{2,0}\}$, $\{H_{1,0},H_{3,0}\}$, $\{H_{2,0},H_{3,0}\}$ and $\{H_{1,0},H_{2,0},H_{3,0}\}$ cannot belong to $\scheme$.
Therefore, $\scheme$ is $\{\{H_{1,0}\},\{H_{2,0}\},\{H_{3,0}\}\}$ and hence $\Psi(l)$ equals $\Psi(l,\{\{H_{1,0}\},\{H_{2,0}\},\{H_{3,0}\}\})$.
Since $u$ does not belong to the last-mentioned basin of attraction, the social outcome $l$ is not an u-local optimum.

\begin{rem}
The social rule $\succ$ is the smallest one that has a global optimum, an u-local optimum and a local optimum.
Indeed, a social rule with less than eight nodes can have at most two features;
if there is only one feature, there is at most one local optimum (which is actually a global optimum);
if there are two features assuming $m_1$ and $m_2$ values (the smaller of which is $2$), by virtue of Theorem~\ref{teo:max_free} there are at most $\min\{m_1,m_2\}=2$ local optima.
\end{rem}

\begin{small}

\end{small}

\end{document}